\documentclass[english, 12pt]{amsart}

\usepackage{tikz-cd}
\usepackage{aliascnt}
\usepackage{amsfonts}
\usepackage{mathrsfs}
\usepackage{amsthm}
\usepackage{amsmath,amssymb}
\usepackage[all,poly,knot]{xy}
\usepackage{amsfonts}
\usepackage{color}
\usepackage{hyperref}
\usepackage{comment}
\usepackage{csquotes}
\usepackage{amscd}
\usepackage{mathtools}
\usepackage{dsfont}

\usepackage[charter]{mathdesign}

\DeclareFontFamily{U}{dutchcal}{\skewchar\font=45 }
\DeclareFontShape{U}{dutchcal}{m}{n}{<-> s*[1.0] dutchcal-r}{}
\DeclareFontShape{U}{dutchcal}{b}{n}{<-> s*[1.0] dutchcal-b}{}
\DeclareMathAlphabet{\mathlcal}{U}{dutchcal}{m}{n}
\SetMathAlphabet{\mathlcal}{bold}{U}{dutchcal}{b}{n}

\usepackage{contour}
\usepackage{ulem}
\usepackage{amsmath,calligra,mathrsfs}
\DeclareMathOperator{\sheafhom}{\mathscr{H}\text{\kern -3pt {\calligra\large om}}\,}
\DeclareMathOperator{\sheafend}{\mathscr{E}\text{\kern -3pt {\calligra\large nd}}\,}

\theoremstyle{plain}
\newtheorem{thmx}{Theorem}

\newtheorem{thm}{Theorem}[section]

\newtheorem{corx}[thmx]{Corollary}
\newtheorem{lem}[thm]{{Lemma}}

\newtheorem{prop}[thm]{Proposition}

\newtheorem{conjx}[thmx]{Conjecture}

\newtheorem{exa}[thm]{Example}

\newtheorem{rmk}[thm]{Remark}

\theoremstyle{definition}
\newtheorem{definition}[subsubsection]{Definition}

\newtheorem{definition-proposition}[subsubsection]{Definition-Proposition}

\newtheorem{theorem}[subsubsection]{Theorem}

\newtheorem{remark}[subsubsection]{Remark}

\newtheorem{proposition}[subsubsection]{Proposition}
\newtheorem*{proposition*}{Proposition}
\newtheorem{lemma}[subsubsection]{Lemma}

\newtheorem*{quest}{Question}

\newtheorem*{theorem*}{Theorem}

\makeatletter
\newcommand{\namelabel}[1]{%
  \phantomsection
  \renewcommand{\@currentlabel}{#1}
  \label{#1}
}
\makeatother

\newcommand{\thistheoremname}{}
\newtheorem*{genericthm*}{\thistheoremname}
\newenvironment{namedthm*}[1]{\renewcommand{\thistheoremname}{#1}%
	\begin{genericthm*}}
	{\end{genericthm*}}

\theoremstyle{remark}

\numberwithin{equation}{section}

\def\At{\mathrm{At(P)}}

\usepackage[top=1in, bottom=1in, left=1in, right=1in]{geometry}

 \usepackage[hyperpageref]{backref}
 
\usepackage{xcolor}
\hypersetup{
	colorlinks,
	linkcolor={red!50!black},
	citecolor={blue!50!black},
	urlcolor={blue!80!black}
}

\contourlength{0.8pt}

\newcommand{\p}{\partial}
\newcommand{\om}{\omega}
\newcommand{\cc}{\frac{1}{2}}
\newcommand{\bg}{g^{\star_{h_0}}}
\newcommand{\Mdol}{\mathcal M_{\mathrm{Dol}}(\mathcal X/S)}
\newcommand{\Ch}{D_{h_0}}
\newcommand{\Ehol}{\mathcal A^{0,0}(\mathrm{End}E\otimes K_{X_0})}
\newcommand{\Eahol}{\mathcal A^{0,1}(\mathrm{End}E)}
\newcommand{\Hy}{\mathbb H^1(X_0,(\mathrm{End}E,\mathrm{ad}(\theta)))}
\newcommand{\Dol}{\mathrm{Dol}}
\newcommand{\End}{\mathrm{End}}
\newcommand{\SL}{\mathrm{SL}(n,\mathbb C)}
\newcommand{\tad}{\mathrm{ad}(\theta)}
\newcommand{\Ker}{\mathrm{Ker}}

\begin{document}
\title[Non-Abelian KS map and non-existence of holomorphic isomonodromic deformation]{Non-Abelian Kodaira-Spencer Map and non-existence of holomorphic isomonodromic deformation of Higgs bundles over Teichm\"uller spaces}

\author[Tianzhi Hu]{Tianzhi Hu}
\address{ School of Mathematics and Statistics, Wuhan University, Luojiashan, Wuchang, Wuhan, Hubei, 430072, P.R. China}
\email{hutianzhi@whu.edu.cn}

\author{Ruiran Sun}
 \address{School of Mathematical Sciences, Xiamen University, Xiamen 361005, China}
\email{ruiransun@xmu.edu.cn}

\author[Kang Zuo]{Kang Zuo}
\address{ School of Mathematics and Statistics, Wuhan University, Luojiashan, Wuchang, Wuhan, Hubei, 430072, P.R. China; Institut f\"ur Mathematik, Universit\"at Mainz, Mainz, Germany, 55099}
\email{zuok@uni-mainz.de}
\begin{abstract}
We define the isomonodromic deformation of a Higgs bundle on a compact Riemann surface via the Hitchin--Simpson correspondence and the isomonodromic deformation of the associated local system. This construction yields a real-analytic section of the relative Dolbeault moduli space and hence a real-analytic foliation, generalizing the Betti foliation arising from the Betti map in the study of abelian schemes. We give cohomological expressions for the holomorphic and anti-holomorphic derivatives of the isomonodromic deformation and use the latter to extend the classical non-abelian Kodaira--Spencer map. 

We prove that if the isomonodromic deformation of a graded Higgs bundle is non-holomorphic, then the deformed Higgs field is non-nilpotent. We also give a short new proof of the non-existence of holomorphic isomonodromic deformations for generic Higgs bundles over Teichm\"uller space $\mathcal T_g$, previously established in \cite{biswas}. This shows that global holomorphicity imposes strong restrictions on the initial Higgs bundle. Motivated by this observation, we prove, under suitable numerical conditions, that non-nilpotent or non-unitary Higgs bundles have non-holomorphic isomonodromic deformations over $\mathcal T_g$. These results may be viewed as analogues of the Landesman--Litt finite-image theorem for MCG-finite representations \cite{LL}.

This paper synthesizes and refines our two earlier preprints \cite{HSZ,HSZII} (arXiv:2511.14272 and arXiv:2512.15478), and makes further progress on the non-existence problem for holomorphic
isomonodromic deformations of higher rank Higgs bundles over Teichm\"uller spaces.
\end{abstract}


\maketitle

\setcounter{tocdepth}{1}
\tableofcontents

\section{Introduction}

\subsection{Nonabelian Hodge correspondence}
Let $X$ be a compact K\"ahler manifold. The celebrated Nonabelian Hodge correspondence  developed primarily by Donaldson \cite{Don} and Uhlenbeck-Yau \cite{UY} in the vector bundle setting and by Hitchin \cite{Hit}, Corlette \cite{Corl} and Simpson \cite{Simp88,Simp92} in the Higgs bundle setting, establishes a real analytic homeomorphism between the moduli space of semisimple $\mathbb C$-local systems and the moduli space of polystable Higgs bundles with vanishing Chern classes. This correspondence is denoted by
$$\Psi:\mathcal M_{\mathrm{B}}(X)\stackrel{\sim}{\longrightarrow} \mathcal M_{\mathrm{Dol}}(X),$$
where we consider $\mathrm{GL}(n,\mathbb C)$ local systems and Higgs bundles unless otherwise stated. In this paper, we work primarily in the Higgs bundle setting.


On the other hand, the classical Riemann-Hilbert correspondence over $X$ provides a biholomorphic map $\mathcal M_{\mathrm{B}}(X)\simeq\mathcal M_{\mathrm{DR}}(X)$, where the latter is the moduli space of semisimple $\mathrm{GL}(n,\mathbb C)$ flat bundles over $X$.

\subsection{Non-abelian Hodge theory for a projective family}\label{nabHT}
For a smooth projective family $f:\mathcal X\to S$ over a quasi-projective base, with a central fiber $X_0$ over $0\in S$, Griffiths \cite{Gri69, Gri70} developed the theory of variation of Hodge structures, which is crucial in deformation theory and Hodge theory (see \cite{CMP}). 

Inspired by Deligne, Simpson \cite{Simp95} introduced a non-abelian analogue of classical variation of Hodge structures. Deligne's twistor space is defined as 
$$\mathcal P:\mathcal M_{\mathrm{Del}}(\mathcal X/S)\to S\times\mathbb P^1,$$ 
whose fiber at $S\times \{1\}$ is $\mathcal M_{\mathrm{DR}}(\mathcal X/S)$ and at $S\times \{0\}$ is $\Mdol$. Both relative moduli spaces are generally not smooth, even if $\mathcal X/S$ is a smooth family. In this paper, all tangent bundles of these relative moduli spaces are defined only at their smooth points.

In \cite{Simp95}, the \textbf{non-abelian Gauss-Manin connection} $\nabla_{GM}$ is defined via the isomonodromic deformation on $\mathcal M_{\mathrm{Del}}(\mathcal X/S)|_{S\times\mathbb G_m}$, and the \textbf{non-abelian Hodge filtration} is the $\mathbb G_m$-action on $\mathcal M_{\mathrm{Del}}(\mathcal X/S)|_{S\times\mathbb G_m}$. Consequently, $\Mdol$ can be interpreted as the '\textbf{non-abelian Hodge bundle}' of $\mathcal M_{\mathrm{DR}}(\mathcal X/S)$ due to the Rees module construction. 

Let $\mathcal M^{\mathrm{gr}}_{\mathrm{Dol}}(\mathcal X/S)\subset\Mdol$ denote the moduli of graded Higgs bundles. In \cite[Lemma 4.1]{Simp10}, Simpson proved that for any $(E,D)\in \mathcal M_{\mathrm{Del}}(X_0)|_{\mathbb A^1}$, 
$$\lim\limits_{\lambda\in\mathbb G_m,\ \lambda\to 0}(E,\lambda\cdot D)\text{ exists and is contained in } \mathcal M^{\mathrm{gr}}_{\mathrm{Dol}}(X_0).$$ 

Let $p_0:\mathcal M^{\mathrm{gr}}_{\mathrm{Dol}}(\mathcal X/S)\to S$ be the projection and $\rho_{KS}$ be the Kodaira-Spencer map of $\mathcal X/S$. 

As proposed in \cite{Simp95}, and by taking the residue at $S\times \{0\}$ of the non-abelian Gauss-Manin connection (as in \cite{CT} and \cite{FS}), we obtain the following \textbf{non-abelian Kodaira-Spencer map} (or non-abelian Higgs field): 
\begin{equation}\label{Theta_ks}\begin{aligned}\Theta_{KS}:p_0^\ast TS&\to T({\Mdol/S})\\ \{(E,\bar\p,\theta),\frac{\p}{\p t}\}&\mapsto \theta_\ast\circ\rho_{KS}(\frac{\p}{\p t}),
\end{aligned}\end{equation}
where $(E,\bar\p,\theta)\in\mathcal M_{\mathrm{Dol}}^{\mathrm{gr}}(X_s)$, $\frac{\p}{\p t}\in T_sS,$ and 
\begin{align*}\theta_\ast:H^1(X_s,T_{X_s})\to\mathbb H^1(X_s,(\mathrm{End}E,\mathrm{ad}(\theta)))=T_{(E,\bar\p,\theta)}\mathcal M_{\mathrm{Dol}}(X_s)
\end{align*}
is the induced map from the morphism of two complexes (each column is a complex)
\[
\begin{tikzcd}
0 \arrow[r, ]                                                   & \mathrm{End} E \otimes  K_{X_s} \\
T_{X_s} \arrow[r, "\theta"] \arrow[u, ] & \mathrm{End} E \arrow[u, "{\tad}"]              
\end{tikzcd}
\]
We remark that $\theta_\ast$ is defined for any Higgs bundle, not necessarily the graded one.

Let $p:\Mdol\to S$ be the projection. We offer a different approach to derive and naturally extend the non-abelian Kodaira-Spencer map $\Theta_{KS}$. 

By performing the \textbf{isomonodromic deformation of a Higgs bundle} (explained in section \ref{1-3}), we obtain a real analytic section $\sigma:S\to\Mdol$. Its anti-holomorphic derivative at $\sigma(s)$ along $\p/\p t\in T_s S$ is given by
\begin{equation}\label{Phi_ks}\begin{aligned}\Phi^{0,1}_{KS}:p^\ast TS&\to T^{1,0}({\Mdol/S})\xrightarrow{\text{conjugate}} T^{0,1}({\Mdol/S})\\ \{(E,\bar\p,\theta),\frac{\p}{\p t}\}&\mapsto \theta_\ast\circ\rho_{KS}(\frac{\p}{\p t}) \xrightarrow{\text{conjugate}} \overline{\theta_\ast\circ\rho_{KS}(\frac{\p}{\p t})},
\end{aligned}\end{equation}

Note that $\Theta_{KS}$ in \eqref{Theta_ks} is defined only for a graded Higgs bundle, while $\Phi^{0,1}_{KS}$ is defined for any Higgs bundle. Therefore, $\Phi_{KS}^{0,1}$ \textbf{extends} $\Theta_{KS}$ to any Higgs bundle in $\Mdol$.

Esnault and Kerz previously asked
whether the anti-holomorphic derivative $\Phi_{KS}^{0,1}$ has such a cohomology expression \eqref{Phi_ks}. We are grateful to them for highlighting this perspective.
\subsection{Isomonodromic deformation of a Higgs bundle and its derivative}\label{1-3}
As our focus is on the local properties of isomonodromic deformation, we assume $S$ is a germ of some variety. An isomonodromic section $\tau:S\to\mathcal M_{\mathrm{DR}}(\mathcal X/S)\cong\mathcal M_{\mathrm{B}}(\mathcal X/S)$ represents an isomonodromic deformation for a given flat vector bundle (or a local system). 

The Nonabelian Hodge correspondence $\Psi$ in \eqref{DUYHS} provides the \textbf{isomonodromic deformation section} of Higgs bundle, defined by:
$$\sigma:=\Psi\circ\tau:S\to \Mdol.$$

We refer to this $\sigma$ as an \textbf{isomonodromic deformation} of a Higgs bundle $\sigma(0)=(E,\bar\p,\theta)\in\mathcal M_{\mathrm{Dol}}(X_0)$. Instead of being holomorphic, this section $\sigma$ is smooth \cite[Theorem 4.7]{Tos} and real analytic \cite[Theorem 4.23]{CTW}. All isomonodromic sections of $\Mdol$ define a real analytic foliation on $\Mdol$, called the \textbf{isomonodromic foliation} on $\Mdol$.

The Betti map of an abelian scheme, introduced in \cite[Section 2.1]{CMZ}, is a valuable tool in studying Diophantine problems, including the distribution of torsion values in an abelian scheme \cite{CMZ, ACZ} and the geometric Bogomolov conjecture \cite{CGHX}. We observe that (in Section \ref{Bettimap}):
\begin{itemize}
\item The Betti map of an abelian scheme $\mathcal A\to S$ is equivalent to a real analytic map 
$$\mathcal B:\mathcal A\to \mathcal M_{\mathrm B}(\mathcal A/S,\mathrm U(1))\cong S\times \text{a fixed real torus}.$$
defined in \eqref{Bet} by the unitary monodromy representation.
\item When considering the relative Jacobian $\mathrm{Jac}(\mathcal X/S)$ of a family of compact Riemann surfaces $\mathcal X/S$, the Betti map is a special case of the Nonabelian Hodge correspondence:
\[\begin{tikzcd}\mathrm{Jac}(\mathcal X/S)\arrow[r,"\mathcal B"]\arrow[d,equals]&\mathcal M_{\mathrm B}(\mathrm{Jac}(\mathcal X/S)/S,\mathrm U(1))\arrow[d,equals]\\
\mathcal M_{\mathrm {Dol}}(\mathcal X/S,\mathrm U(1))\arrow[r]&\mathcal M_{\mathrm B}(\mathcal X/S,\mathrm U(1))\end{tikzcd}\]
The Betti foliation on $\mathrm{Jac}(\mathcal X/S)$, defined by the level sets of the Betti map, is precisely the isomonodromic foliation on $\mathcal M_{\mathrm {Dol}}(\mathcal X/S,\mathrm U(1))$.
\end{itemize}

By \cite[Proposition 2.1]{CMZ} and \cite[Section 2.1]{CGHX}, any level set of a Betti map is a holomorphic section of the abelian scheme. However, an isomonodromic deformation $\sigma:S\to\Mdol$ is not always holomorphic.
\begin{quest} When is an isomonodromic deformation $\sigma:S\to\Mdol$ holomorphic?\end{quest}
\begin{flushleft}\textbf{Setup:} We focus on the smooth projective family of compact Riemann surfaces $\mathcal X/S$. Nevertheless, the main ingredient in our paper, i.e. using the deformation of the harmonic metric to study the deformation of the Higgs bundle, should work for any higher dimensional family.\end{flushleft}

After reviewing the deformation theory of a Higgs bundle in $\Mdol$ using the theory of the Atiyah bundle (developed in \cite{Ati}) and studying the first-order deformation of the harmonic metric, we arrive at the following {\bf cohomology  representation} for the tangent map of $\sigma$, which answers the question regarding holomorphicity. 
\begin{thmx}\label{thm-A}
(i) The \textbf{holomorphic derivative} $\Pi^{1,0}\sigma_\ast(\frac{\p}{\p t})$ of the isomonodromic section $\sigma:S\to\Mdol$ is given by
\begin{align*}\Phi^{1,0}_{KS}:T_0S&\to T^{1,0}_{\sigma(0)}\Mdol\\
\frac{\p}{\p t}&\mapsto [(-\cc\Ch^{1,0}g,-\cc[\theta^{\star_{h_0}},g],\eta)],\end{align*}
where this triple is a deformation class defined in Proposition \ref{holtan}. Here $[\eta]:=\rho_{KS}(\p/\p t)$ is the Kodaira-Spencer class, and $g$ is a smooth endomorphism of $E$ satisfying the PDE \eqref{equ_g} in Proposition \ref{pde} such that $g$ is uniquely determined by $\frac{\p}{\p t}$ and the initial Higgs bundle $\sigma(0)\in\Mdol$;\\
(ii) The \textbf{anti-holomorphic derivative} $\Pi^{0,1}\sigma_\ast(\frac{\p}{\p t})$ for the isomonodromic section $\sigma:S\to\Mdol$ is 
\begin{align*}\Phi^{0,1}_{KS}:T_0S\stackrel{\rho_{KS}}{\longrightarrow}H^1(T_{X_0})\stackrel{\theta_\ast}{\longrightarrow}T_{\sigma(0)}^{1,0}\mathcal M_{\mathrm{Dol}}(X_0)\stackrel{\text{conjugate}}{\longrightarrow}T_{\sigma(0)}^{0,1}\mathcal M_{\mathrm{Dol}}(X_0)\hookrightarrow T_{\sigma(0)}^{0,1}\Mdol.
\end{align*}
\end{thmx}
We remark that \cite[Proposition 4.6]{Tos} also gives the first order deformation of the isomonodromic deformation, but  he does not distinguish the holomorphic and anti-holomorphic parts and does not give a cohomology interpretation. Moreover, the Levi form of the energy functional of $\sigma$ is concerned in \cite{Tos} and \cite{CTW}. There they prove the constancy of the energy functional is equivalent to the holomorphicity of $\sigma$.

We refer to the above derivatives $\Phi^{1,0}_{KS}$ and $\Phi^{0,1}_{KS}$ as the \textbf{non-abelian Kodaira-Spencer maps} because they arise directly from the isomonodromic deformation and the Nonabelian Hodge correspondence.

The partial differential equation (PDE) that $g$ satisfies is a linear non-homogeneous second-order elliptic PDE. The complex conjugation of $\Phi^{0,1}_{KS}$ is indeed a sheaf morphism $ p^\ast TS\to T(\Mdol/S)$. But the holomorphic part $\Phi^{1,0}_{KS}$ seems to be \textbf{transcendental} even in the rank one case, i.e. $\mathcal M_{\mathrm{Dol}}(\mathcal X/S,\mathbb C^\ast)$.

\subsection{Holomorphicity and rigidity} We shall explain in this section: the holomorphicity controls the deformation rigidity of $\sigma$. 

A graded Higgs bundle over $X_0$ is equivalent to the Hodge bundle of a polarized $\mathbb C$-variation of Hodge structures over $X_0$ (see \cite{CMP}). Simpson \cite{Simp92, Simp10} also proved that graded Higgs bundles are precisely those Higgs bundles invariant under the $\mathbb C^\ast$ action: 
$(E,\bar\p,\theta)\stackrel{\lambda\in\mathbb C^\ast}{\longrightarrow}(E,\bar\p,\lambda\theta).$

We first consider the isomonodromic deformation $\sigma:S\to\Mdol$ of a graded Higgs bundle $\sigma(0) \in\mathcal M_{\mathrm{Dol}}(X)$.

\begin{quest} Is the isomonodromically deformed Higgs bundle still graded?\end{quest}

This question probes the commutativity between the $\mathbb C^\ast$ action and the isomonodromic deformation of a Higgs bundle. The answer is generally no: \cite[Theorem 1.4]{KYZ} shows that if  $\Phi_{KS}^{0,1}$ is non-zero, then $\sigma(t)$ is not a graded Higgs bundle for $t$ near $0$. 

We can prove the following stronger result by applying Theorem \ref{thm-A}:
\begin{thmx}[=Theorem~\ref{nonnil}]\label{thm E} Assume (a) $\sigma(0):=(E,\bar\p,\theta)$ is a graded Higgs bundle on $X$, and \\(b) the non-abelian Kodaira-Spencer map $\Phi^{0,1}_{KS}$ is non-zero at $\sigma(0)$ for some $\frac{\p}{\p t}\in T_0S$.\\
Then the isomonodromic deformation of $(E,\bar\p,\theta)$ along the direction $\frac{\p}{\p t}\in T_0S$ is non-nilpotent.
\end{thmx}
In the proof of Theorem \ref{thm E}, we show that the image of $\Phi^{0,1}_{KS}$ is contained in the horizontal direction (defined in \cite[section 1.3]{moch24}) of the Hitchin fibration map.

Motivated by Theorem~\ref{thm E}, the equivalence between holomorphicity and gradedness has been further clarified by Hu, Sun, Yang, and Zuo \cite{HSYZ}. Starting from a graded Higgs bundle associated with a polarized \(\mathbb C\)-variation of Hodge structures, they prove that the non-abelian Noether--Lefschetz locus, namely the locus on which the isomonodromically deformed Higgs bundle remains graded, is precisely the maximal complex analytic subvariety of the base on which the real analytic section \(\sigma\) becomes holomorphic. 

\vspace{.5cm}

Similar rigidity holds for \textbf{nilpotent Higgs bundles}. Let $\sigma(0) \in\mathcal M_{\mathrm{Dol}}(X)$ be a nilpotent Higgs bundle. One may ask whether on the complex analytic subvariety $U\subset S$ passing through $0$ such that $\sigma|_U$ is holomorphic, the restriction $\sigma|_U$ still represents nilpotent Higgs bundles? Hu, Sun, Yang and Zuo \cite[Theorem 1.10]{HSYZ} prove that this holds when $\sigma(0)$ is generically regular nilpotent, which gives a partial result on this rigidity.

One can also explore the {\bf distribution} of nilpotent Higgs bundles along the isomonodromic deformation over the Teichm\"uller space $\mathcal T_g$:
\begin{conjx} We say the isomonodromic section $\sigma:\mathcal T_g\to \mathcal M_{\mathrm{Dol}}(\mathcal X/\mathcal T_g)$ has nilpotent loci $\mathcal N\subset \mathcal T_g$ if $\mathcal N$ is a union of all maximal irreducible real analytic subvarieties $S_i\subset \mathcal T_g,\ i\in I$, such that $\sigma|_{S_i}$ is always nilpotent. We expect that the family $\{S_i\}_{i\in I}$ is {\bf locally finite} in
$\mathcal T_g$; that is, for every $s\in\mathcal T_g$, there exists a
Euclidean open neighborhood $U$ of $s$ such that $\{i\in I : U\cap S_i\neq\varnothing\}$ is a finite set.
\end{conjx}

\subsection{Global holomorphicity fails in general} Let $\mathcal T_g$ be the Teichm\"uller space of genus $g\geq 2$ compact Riemann surfaces and let $\mathcal X/\mathcal T_g$ be its universal family. Since $\mathcal T_g$ is contractible, any isomonodromic section $\sigma:\mathcal T_g\to \mathcal M_{\mathrm{Dol}}(\mathcal X/\mathcal T_g)$ is well-defined. We call $\sigma$ \textbf{globally holomorphic} if it is a holomorphic section over $\mathcal T_g$. Now we focus on the {\bf non-existence} of globally holomorphic isomonodromic section, which was first considered in \cite{biswas}.

\begin{thmx}\label{main B} For a generic Higgs bundle $(E,\bar\p,\theta)$ with $\mathrm{rank}\geq2$ in $\mathcal M_{\text{Dol}}(\mathcal X/\mathcal T_g)$, its isomonodromic deformation fails to be holomorphic along any direction of $\mathcal T_g$. 
\end{thmx}

We note that Theorem \ref{main B} corresponds to the claim in \cite[Theorem 3.3]{biswas}.  Their proof relies heavily on \cite[Theorem 1.6]{Tos}. Here, we present a
different and simpler proof in Section~\ref{sec_generic} utilizing the criterion Theorem \ref{thm-A} (ii) alongside properties of
the associated spectral curve.

We also establish non-existence results for low-rank non-unitary Higgs bundles, specifically for ranks 2 and 3.

\begin{thmx}\label{main C}(i) For any rank 2 polystable non-unitary Higgs bundle $(E,\bar\p,\theta)$, its isomonodromic deformation fails to be holomorphic over $\mathcal T_g$.\\
(ii) For any rank 3 stable non-unitary Higgs bundle $(E,\bar\p,\theta)$, its isomonodromic deformation fails to be holomorphic over $\mathcal T_g$  (when $g\geq 3$, this is also true for any rank 3 polystable non-unitary Higgs bundle).
\end{thmx}

Theorem \ref{main C} (i) was first stated in \cite[Theorem 3.4]{biswas}. However, by applying the criterion Theorem \ref{thm-A} (ii), our approach offers a simpler proof. Indeed, we show in Section~\ref{sec_rk2} and Section~\ref{sec_rk3}
 that $\theta_\ast$ is always non-zero under the assumptions of Theorem \ref{main C}. 

Motivated by Theorem \ref{main B} and Theorem \ref{main C}, we make the following conjecture:
\begin{conjx}\label{conj-F}For $g\geq 3$ and a non-nilpotent Higgs bundle $(E,\bar\p,\theta)$ in $\mathcal M_{\text{Dol}}(\mathcal X/\mathcal T_g)$, its isomonodromic deformation fails to be holomorphic over $\mathcal T_g$.
\end{conjx}
This conjecture may be viewed as a Dolbeault-side analogue of the following theorem on representations of surface groups \cite[Theorem~1.2.1]{LL}:
\begin{theorem*}[Landesman--Litt]
Let $\Sigma_g$ be a smooth oriented surface of genus $g$ and 
\[
\rho:\pi_1(\Sigma_{g})\longrightarrow\mathrm{GL}_n(\mathbb C)
\]
be a representation which has finite orbits under the mapping class group of $\Sigma_g$ (called \textbf{MCG-finite}). If $n=\mathrm{rank}\ \rho <\sqrt{g+1},$ then $\rho$ has finite image. This result also holds for punctured surfaces.
\end{theorem*}A MCG-finite semisimple representation corresponds to an algebraic isomonodromic deformation after a finite base change. Hence MCG-finite implies the holomorphicity as an isomonodromic deformation of the Higgs bundle. The numerical condition in their theorem is indispensable: as explained in \cite[Section~10]{LL}, there are MCG-finite representations of large rank with infinite image. Those examples are both $\mathbb C$-PVHS and hence nilpotent as Higgs bundles. 

Conjecture~\ref{conj-F} is equivalent to the statement that global holomorphicity implies nilpotency. Hence it also takes the above examples with large rank into consideration. Using Theorem~\ref{thm-A}(ii) together with a Clifford-type estimate, we obtain the following numerical result.
\begin{thmx}\label{thm-G}
Let $(E,\bar\p,\theta)$ be a rank-$n$ polystable Higgs bundle on a compact Riemann surface of genus $g\geq2$, and let $\sigma$ be its isomonodromic deformation over $\mathcal T_g$.
\begin{itemize}
\item[(i)] If $g>(n-1)^2$ and $(E,\bar\p,\theta)$ is non-unitary, then $\sigma$ is not globally holomorphic.
\item[(ii)] If $g>(n-2)^2+1$ and $\theta$ is non-nilpotent, then $\sigma$ is not globally holomorphic.
\item[(iii)] Moreover, if the spectral curve of $\theta$ is reduced, $g\geq n$, and $(E,\bar\p,\theta)$ is non-unitary, then $\sigma$ is not globally holomorphic.
\end{itemize}
\end{thmx}

The proof of Theorem~\ref{thm-G} will be given in Section~\ref{sec_globolhol}. As explained in Section~\ref{sec_counterexample}, Conjecture~\ref{conj-F} fails in genus two, so we assume $g\geq3$. 

\vspace{.3cm}

As an arithmetic consequence of Theorem~\ref{thm-G}(ii), we obtain the following finite-monodromy statement.

\begin{corx}\label{cor-H}
Let $K\subset \mathbb C$ be a number field with ring of integers $\mathcal O_K$, and let
\[
\rho\colon \pi_1(\Sigma_g)\longrightarrow \mathrm{GL}_n(\mathcal O_K)
\]
be a semisimple representation. Assume that $g>(n-1)^2$.
If the isomonodromic deformations of the Higgs bundles associated with
all Galois conjugates of $\rho$ are globally holomorphic, then $\rho$
has finite image.
\end{corx}
\begin{proof}By Theorem~\ref{thm-G}(i), every Galois conjugate of $\rho$ is unitary.
Consequently, the image of $\rho$ is bounded under every archimedean
embedding of $K$. Also, the image of $\rho$ in the product of all archimedean realizations is
discrete. A discrete and bounded subset is finite, and hence $\rho$ has
finite image.
\end{proof}

\vspace{.5cm}

\noindent\textbf{Acknowledgements:}
H\'el\.ene Esnault and Moritz Kerz previously asked whether the anti-holomorphic component discussed in \cite[Theorem A (ii)]{HSZ}
admits such a cohomological expression. We are grateful to them for highlighting this perspective. 
We also would like to thank Esnault and Kerz for drawing our attention to the paper \cite{biswas}. We would like to thank Xin L\"u, Takuro Mochizuki, Sheng Rao, Richard Wentworth, Jinbang Yang and Runze Zhang for their  helpful suggestions and discussions.

\section{Preliminaries}
Now we review some basic theory and notation. For a background on Hodge theory and Higgs bundles, one may refer to \cite{CMP}. For a survey paper on harmonic maps and Higgs bundles, one may refer to \cite{Li}.
\subsection{Local system and vector bundle}
Let \( X_0 \) be a compact Riemann surface of genus $g\geq 2$ with a K\"ahler form $\omega_0$. For any holomorphic (smooth) vector bundle $E$ over $X_0$, let $\mathcal A^{p,q}(E)$ be the space of smooth $(p,q)$ sections of $E$ on $X_0$.

\hfill

\noindent
\textbf{Flat bundle}: a tuple \( (E, D) \), where \( E \) is a holomorphic vector bundle on $X_0$ with a flat holomorphic connection $D:\mathcal O(E)\to \mathcal O(E\otimes  K_{X_0})$. After tensoring $C^{\infty}(X_0)$, $(E,D)$ is a smooth flat vector bundle with a flat connection $D:\mathcal A^0(E)\to\mathcal A^1(E)$. In this paper, a flat bundle refers to a smooth flat bundle unless otherwise stated.

\hfill

\noindent
\textbf{$\mathbb C$-local system}: a locally constant sheaf over $X_0$ with stalk \( V \), where $V$ is a complex vector space. For a local system \( \mathcal V\) over \( X_0 \), $ \mathcal V \otimes_{\mathbb{C}} \mathscr C^{\infty}_{X_0}$ is a locally free sheaf of \( \mathscr C^{\infty}_{X_0} \)-modules, which is a flat smooth vector bundle. Similarly, $\mathcal V \otimes_{\mathbb{C}} \mathcal{O}_{X_0}$ is a locally free sheaf of \( \mathcal{O}_{X_0} \)-modules, which is a flat holomorphic vector bundle.  

\hfill

\noindent
\textbf{Chern connection}: given a holomorphic vector bundle \( (E, \bar\p) \) with a Hermitian metric \( h \), there exists a unique unitary connection \( D_{h} \) such that \( D_{h}^{0,1} = \bar\p \), which is called the Chern connection. Let $F(D_h):=D_h\circ D_h$ be the Chern curvature, which is a (1,1) form of $\mathrm{End}E$.

\hfill

\noindent
\textbf{Hermitian metric on \( \mathcal A^1(\mathrm{End} E) \)}: let $(E,h)$ be a Hermitian vector bundle. For any \( \varphi, \psi \in \mathcal A^0(\mathrm{End} E) \), $\alpha,\beta\in \mathcal A^1(X_0)$  
\[
\langle \varphi \otimes \alpha, \psi \otimes \beta \rangle := \mathrm{tr}(\varphi \psi^{\star_{h}}) \cdot \Lambda_{\omega_0}(\alpha \wedge *\beta)
\]  
where $*$ is the Hodge star operator with respect to $\omega_0$, $\star_h$ is the Hodge star operator with respect to $h$, and $\Lambda_{\omega_0}$ is the contraction by $\omega_0$.

\subsection{Harmonic metric}
Given a flat vector bundle \( (E, D) \), and any Hermitian metric $h$ on it,  
we have the following unique decomposition
\[
D = D_h + \psi_h
\] 
where \( D_h \) is unitary, and \( \psi_h \) is self-adjoint, decomposed by type (1,0) and (0,1):  
\begin{align*}
    D  = D_h + \psi_h  = D_h^{1,0} + D_h^{0,1} + \theta + \theta^{\star_h}.
\end{align*}

Define the energy functional 
\[
E(h) := \int_{X_0} \langle \psi_h, \psi_h \rangle \, \omega_0.
\]  

If \( h \) is a minimal point of \( E(h) \), we call it a \textbf{harmonic metric}. Such $h$ satisfies the following equation
$$D_h^{\star_h}\psi_h=0.$$

If a harmonic metric $h$ on $(E,D)$ exists, then  
\( (E, D_h^{0,1}, \theta) \) is a Higgs bundle.

\begin{theorem*}[Donaldson, Corlette and Uhlenbeck-Yau]
    A semisimple flat bundle $(E,D)$ admits a harmonic metric. 
\end{theorem*}

\subsection{Higgs bundle}
A Higgs bundle over $X_0$ is a triple \( (E, \bar{\partial}, \theta) \), where \( E \) is a holomorphic vector bundle on $X_0$ with a holomorphic structure \( \bar{\partial} \), and \( \theta \in \mathcal A^{1,0}(\mathrm{End} E) \) satisfies \( \bar{\partial}\theta = 0 \) and \( \theta \wedge \theta = 0 \). (When $X_0$ is a compact Riemann surface, the condition $\theta\wedge\theta=0$ is automatically satisfied.)

We say $h$ is a harmonic metric for a Higgs bundle \( (E, \bar{\partial}, \theta) \) if $h$ satisfies the Hitchin-Yang-Mills equation \begin{align}\label{HYM}F(D_h)+[\theta, \theta^{\star_h}]=0.\end{align}

If a harmonic metric $h$ on \( (E, \bar{\partial}, \theta) \) exists, then $(E,D_h+\theta+\theta^{\star_h})$ is a flat bundle.
\begin{theorem*}[Hitchin, Simpson]
    A polystable Higgs bundle with vanishing Chern classes admits a harmonic metric. 
\end{theorem*}

\subsection{Moduli spaces and correspondences} One may refer to \cite{Simp92} and \cite{Simp94I}.
\hfill

\noindent $\mathcal{M}_\mathrm{B}(X_0)$: moduli space of semisimple representations of $\pi_1 (X_0)$ into $\mathrm{GL}(n,\mathbb C)$.

\hfill

\noindent
$\mathcal{M}_{\mathrm{dR}}(X_0)$: moduli space of semisimple $\mathrm{GL}(n,\mathbb C)$ flat vector bundles over $X_0$.

\hfill

\noindent
$\mathcal{M}_{\mathrm{Dol}}(X_0)$: moduli space of polystable $\mathrm{GL}(n,\mathbb C)$ Higgs bundles with vanishing Chern classes on $X_0$.

\hfill

\noindent
\textbf{Riemann-Hilbert correspondence}: we have the biholomorphic map
\begin{align*}
    \mathcal{M}_{\mathrm{B}}(X_0) & \to \mathcal{M}_{\mathrm{DR}}(X_0),\\ 
    [\rho:\pi_1(X_0)\to \mathrm{GL}(n,\mathbb C)] & \mapsto [\widetilde{X_0}\times_\rho \mathbb{C}^n],
\end{align*}
where $\widetilde{X_0}$ is the universal covering of $X_0$.

\hfill

\noindent
\textbf{Nonabelian Hodge correspondence}: we have the real analytic map
\begin{equation}\label{DUYHS}
\begin{aligned}
   \Psi: \mathcal{M}_{\mathrm{DR}}(X_0) & \to \mathcal{M}_{\mathrm{Dol}}(X_0),\\ 
    [(E, D)] & \mapsto [(E, D_h^{0, 1}, \theta)].
\end{aligned}\end{equation}

\subsection{Holomorphic tangent space of $\mathcal{M}_{\mathrm{Dol}}(X_0)$}\label{htdol}
Let $(E,\bar\p,\theta)\in \mathcal M_{\mathrm{Dol}}(X_0)$ be a smooth point. The holomorphic tangent space \( T_{(E, \bar{\partial}, \theta)} \mathcal{M}_{\mathrm{Dol}}(X_0) \) can be characterized by the hypercohomology $\Hy$. We shall use the \textbf{Dolbeault resolution} of the complex $\mathrm{End}E\stackrel{\tad}{\longrightarrow}\mathrm{End}E\otimes K_{X_0}$ to compute this hypercohomology. The Dolbeault resolution is 
\[
\begin{tikzcd}
C^{0,1}:=\Eahol \arrow[r, "{\tad}"]                                                   & C^{1,1}:={\mathcal A^{0,1}(\mathrm{End} E \otimes  K_{X_0})} \\
C^{0,0}:=\mathcal A^{0,0}(\mathrm{End} E) \arrow[r, "{\tad}"] \arrow[u, "\bar\partial"] & C^{1,0}:=\Ehol \arrow[u, "\bar\partial"]              
\end{tikzcd}
\]
Then we have the following truncated complex
\[C^{0,0} \xrightarrow{d^0} C^{1,0} \oplus C^{0,1} \xrightarrow{d^1}  C^{1,1},\]
where 
\begin{gather*}
    d^0(g) = ([\theta,g], \bar{\partial} g) \quad \text{for } \quad g \in C^{0,0},\\
    d^1(\varphi,\psi) = \bar{\partial} \varphi + [\psi,\theta] \quad \text{for } \quad (\varphi,\psi) \in C^{1,0} \oplus C^{0,1}.
\end{gather*}
Then we have the holomorphic tangent space
$$T_{(E, \bar{\partial}, \theta)} \mathcal{M}_{\mathrm{Dol}}(X_0)=\Hy=\frac{\mathrm{Ker\ }d^1}{\mathrm{Im\ }d^0}.$$

\subsection{Relative moduli spaces}
Let \( f: \mathcal{X} \to S \) be a smooth projective family of compact Riemann surfaces over \( S \) with central fiber $X_0$ over $0\in S$. Simpson constructed three relative moduli spaces in \cite{Simp94II}. They are the relative Betti moduli, the relative de Rham moduli, and the relative Dolbeault moduli: 

$$\mathcal{M}_\mathrm{B}(\mathcal{X}/S),\ \mathcal{M}_{\mathrm{DR}}(\mathcal{X}/S),\ \mathcal{M}_{\mathrm{Dol}}(\mathcal{X}/S),$$
whose fibers at $s\in S$ are
$$\mathcal{M}_\mathrm{B}(X_s),\ \mathcal{M}_{\mathrm{DR}}(X_s),\ \mathcal{M}_{\mathrm{Dol}}(X_s),\text{ respectively}.$$
By \cite[Proposition 7.18]{Simp94II}, we have the complex analytic homeomorphism \begin{align}\label{RH}\mathcal{M}_\mathrm{B}(\mathcal{X}/S)\cong \mathcal{M}_{\mathrm{DR}}(\mathcal{X}/S).\end{align}
And by \cite[Theorem 4.23]{CTW} we have the real analytic homeomorphism
$$ \Psi:\mathcal{M}_{\mathrm{DR}}(\mathcal{X}/S)\stackrel{\sim}{\longrightarrow}\Mdol.$$
\subsection{Isomonodromic deformation}\label{Isomd}
By Ehresmann's theorem, locally \( \mathcal{X}/S \) has a \( C^\infty \)-trivialization \( \mathcal{X}|_U \cong U \times X_0 \), where \( U \) is a small neighborhood of $0\in S$. Then for each \( s \in U \), this trivialization gives a diffeomorphism \( F_s : X_s \to X_0 \). To avoid the monodromy of the isomonodromic deformation, we shall assume $S$ is a germ of a positive-dimension variety near $0$.

For \( \rho \in \mathcal{M}_\mathrm B(X_0) \), the pull back \( F_s^* \rho \in \mathcal{M}_\mathrm B(X_s) \), which gives a section of \( \mathcal{M}_B(\mathcal{X}/S) \to S \), called the \textbf{isomonodromic deformation} of \( \rho \) over \( S \), denoted by the section
$$\tau:S\to \mathcal{M}_B(\mathcal{X}/S).$$

The Riemann-Hilbert correspondence gives a biholomorphic map \( \mathcal{M}_\mathrm B(X_s) \cong\mathcal{M}_{\mathrm{DR}}(X_s) \). Under this, we get a section of \( \mathcal{M}_{\mathrm{DR}}(\mathcal{X}/S) \to S \), called the \textbf{isomonodromic deformation of a flat vector bundle} $(\mathcal V,D)\in\mathcal{M}_{\mathrm{DR}}(X_0) $, still denoted by $\tau:S\to \mathcal{M}_{\mathrm{DR}}(\mathcal X/S).$
\begin{definition}(a) Recall the Nonabelian Hodge correspondence $$ \Psi:\mathcal{M}_{\mathrm{DR}}(X_s) \to \mathcal{M}_{\mathrm{Dol}}(X_s).$$ Under this real analytic correspondence, we get a real analytic section $$\sigma:=\Psi\circ \tau:S\to \mathcal{M}_{\mathrm{Dol}}(\mathcal{X}/S),$$ called the \textbf{isomonodromic deformation of a Higgs bundle} $\sigma(0)\in\mathcal M_{\mathrm{Dol}}(X_0).$\\
(b) The \textbf{isomonodromic foliation} of $\Mdol$ is a real analytic foliation generated by all its isomonodromic sections.
\end{definition}
We remark that the isomonodromic deformation can be defined for any smooth projective family of relative dimension $\geq1$.
\subsection{Non-abelian Gauss-Manin connection and non-abelian Kodaira-Spencer map}\label{naGMnaKS}
Let $\pi:\mathcal{M}_{\mathrm{DR}}(\mathcal X/S)\to S$ be the relative de Rham moduli over $S$. We define the non-abelian Gauss-Manin connection on $\mathcal{M}_{\mathrm{DR}}(\mathcal X/S)$ as a horizontal lifting (see \cite{CT} and \cite{FS})
\begin{align*}\nabla_{\mathrm{GM}}: \pi^\ast TS&\to T\mathcal{M}_{\mathrm{DR}}(\mathcal X/S)\\
\{(E,D),\frac{\p}{\p t}\}&\mapsto\tau_\ast(\frac{\p}{\p t}),
\end{align*}
where $(E,D)\in\mathcal M_{\mathrm{DR}}(X_s)$, $\frac{\p}{\p t}\in T_sS$, and $\tau:S\to\mathcal{M}_{\mathrm{DR}}(\mathcal X/S)$ is the isomonodromic deformation of $(E,D)\in\mathcal M_{\mathrm{DR}}(X_s)$ near $s\in S.$ By \eqref{RH}, $\nabla_{\mathrm{GM}}$ is a holomorphic connection. Additionally, $\nabla_{\mathrm{GM}}$ has zero curvature, meaning isomonodromic sections indeed define a holomorphic foliation on $\mathcal{M}_{\mathrm{DR}}(\mathcal X/S)$. Thus, $\mathcal{M}_{\mathrm{DR}}(\mathcal X/S)\to S$ is locally trivial over $S$. By \cite[Section 7]{Simp95}, this $\nabla_{GM}$ extends to $\mathcal M_{\mathrm{Del}}(\mathcal X/S)|_{S\times\mathbb G_m}$. Recall the non-abelian Kodaira-Spencer map $\Theta_{KS}$ defined in \eqref{Theta_ks} of section \ref{nabHT}. 

Now we provide a formula for $\theta_\ast$ using the Dolbeault representative (see \cite[Proposition 2.1]{Zuo} and \cite[Theorem 1.2]{FS}).
\begin{lemma}\label{thetaast}Let $\eta\in \mathcal A^{0,1}(T_{X_s})$ be a representative of $[\eta]\in H^1(X_s,T_{X_s})$. Then $$\theta_\ast([\eta])=[(0,\eta(\theta))]\in\mathbb H^1(X_s,(\mathrm{End}E,\mathrm{ad}(\theta))),$$
where $\eta(\theta)\in \mathcal A^{0,1}(\mathrm{End}E)$ is a contraction by vector fields.
\end{lemma}

\subsection{Betti map and isomonodromic deformation}\label{Bettimap}
Assume our base $S$ is a germ of a variety. Let $\mathcal A\to S$ be an abelian scheme of relative dimension $g$. We can find a basis of relative holomorphic 1-forms of $\mathcal A\to S$, denoted by $\{\omega_1(s),\cdots,\omega_g(s)\}_{s\in S}$. Since all fibers $\mathcal A_s$ have the same topological type, we can take $\gamma_1,\cdots,\gamma_{2g}$ as the common set of generators of their fundamental groups.
\begin{definition}[\cite{CMZ} and \cite{CGHX}] (a) For any $\xi\in \mathcal A_s$, its coordinates $\{\int_0^\xi\omega_i(s)\}_{i=1}^g$ are well defined modulo $\bigoplus\limits_{j=1}^{2g}\mathbb Z\{
\int_{\gamma_j}\omega_i(s)\}_{i=1}^g $, and clearly we have $(b_1(\xi),\cdots,b_{2g}(\xi))\in \mathbb R^{2g}/\mathbb Z^{2g}$, such that $\int_0^\xi=\sum_{j=1}^{2g}b_j(\xi)\int_{\gamma_j}.$ We define the \textbf{Betti map} by 
\begin{align*}b:\mathcal A&\to\mathbb R^{2g}/\mathbb Z^{2g}\\
\xi&\mapsto(b_1(\xi),\cdots,b_{2g}(\xi)),\end{align*}
which is a real analytic map. \\
(b) We say a section $\alpha:S\to \mathcal A$ is a level set of $b$ if $b(\alpha(s))$ is a constant map. Through any $\xi\in\mathcal A_s$, there is a unique level set of $b$ passing through $\xi$. Thus all level sets of $b$ define a real analytic foliation on $\mathcal A$, called the \textbf{Betti foliation}.
\end{definition}
Since $\mathbb R/\mathbb Z\cong \mathrm{U}(1)$, for each $\xi\in \mathcal A_s$, $b(\xi)$ determines an element $\rho_\xi$ in $\mathrm{Hom}(\pi_1(\mathcal A_s),\mathrm{U}(1))=\mathcal M_{\mathrm B}(\mathcal A_s,\mathrm U(1))$ by mapping 
$$\gamma_j\to b_j(\xi),\ j=1,2,\cdots,2g.$$
So we have the following real analytic homeomorphism 
\begin{equation}\label{Bet}\begin{aligned}\mathcal B:\mathcal A&\to\mathcal M_{\mathrm B}(\mathcal A/S,\mathrm{U}(1))\\
\xi&\mapsto\rho_\xi.\end{aligned}\end{equation}
By fixing the generators of $\pi_1(\mathcal A_s),\ \text{for any } s\in S$, $\mathcal M_{\mathrm B}(\mathcal A/S,\mathrm{U}(1))$ is real analytically homeomorphic to $S\times \mathbb R^{2g}/\mathbb Z^{2g}$. Then $b$ and $\mathcal B$ are equivalent via this trivialization.\\
\textbf{Observation:} By the above definition, if $\alpha:S\to\mathcal A$ is a level set of $b$, then $\rho_{\alpha(\cdot)}:S\to\mathcal  M_{\mathrm B}(\mathcal A/S,\mathrm{U}(1))$ is an isomonodromic section. Conversely, any isomonodromic section of $\mathcal M_{\mathrm B}(\mathcal A/S,\mathrm{U}(1))$ defines a level set of the Betti map $b$.

Therefore, the isomonodromic deformation in the relative Betti moduli is a generalization of the level set of the Betti map. 

If $\mathcal X/S$ is a smooth family of compact Riemann surfaces of genus $g\geq2$, then we consider the abelian scheme $\mathrm{Jac}(\mathcal X/S)\to S$. In this case, the map 
$$\mathcal B:\mathrm{Jac}(\mathcal X/S)\to \mathcal M_{\mathrm B}(\mathrm{Jac}(\mathcal X/S)/S,\mathrm{U}(1))=\mathcal M_{\mathrm B}(\mathcal X/S,\mathrm{U}(1))$$ defined in \eqref{Bet} is simply the Nonabelian Hodge correspondence restricted to
$$\mathrm{Jac}(\mathcal X/S)=\mathcal M_{\mathrm {Dol}}(\mathcal X/S,\mathrm{U}(1))\subsetneq \Mdol.$$

Therefore, a level set of the Betti map $b$ on $\mathrm{Jac}(\mathcal X/S)$ is exactly an isomonodromic section of $\mathcal M_{\mathrm {Dol}}(\mathcal X/S,\mathrm{U}(1))$. The Betti foliation coincides with the isomonodromic foliation on $\mathcal M_{\mathrm {Dol}}(\mathcal X/S,\mathrm{U}(1))$.
\section{First order deformation theory for a Higgs bundle}\label{DT}
Let $f:\mathcal X\to S$ be a smooth projective family of compact Riemann surfaces with central fiber $X_0:=f^{-1}(0)$, for $0\in S$. We consider a real analytic family of Higgs bundles over $\mathcal X$ such that for each $s\in S$, this family restricted to $X_s$ is a Higgs bundle on $X_s$.

We have the relative Dolbeault moduli space $p:\Mdol\to S$. We may view a real analytic (or holomorphic) family of Higgs bundles as a real analytic (or holomorphic) section $\sigma: S\to\Mdol$ such that $p\circ\sigma=\mathrm {id}_S$. In the following, we always assume $$\sigma(0)=(E,\bar\p,\theta)\in \Mdol|_0=\mathcal M_{\mathrm{Dol}}(X_0).$$

\begin{lem}\label{realsec}
When $\sigma: S\to\Mdol$ is a real analytic section, we consider its holomorphic and anti-holomorphic derivative along $v\in T_0S\ (=T_0^{1,0}S)$, i.e.,
$$\sigma_\ast(v)=w^{1,0}\oplus w^{0,1}\in T^{1,0}_{\sigma(0)}\Mdol\oplus T^{0,1}_{\sigma(0)}\Mdol.$$
Then $$p_\ast(w^{1,0})=v\text{ and }p_\ast(w^{0,1})=0.$$ In other words, the anti-holomorphic derivative of $\sigma$ along $v$ factors through $T_{\sigma(0)}^{0,1}\mathcal M_{\mathrm{Dol}}(X_0)$. Clearly, when $w^{0,1}=0$, $\sigma$ is holomorphic along $v\in T_0S$.
\end{lem}
\begin{proof} Since $p$ is holomorphic, $p_\ast$ maps $T^{1,0}_{\sigma(0)}\Mdol$ to $T^{1,0}_0S$ and $T^{0,1}_{\sigma(0)}\Mdol$ to $T^{0,1}_0S$. Thus $$(p\circ \sigma)_\ast(v)=p_\ast(w^{1,0})+p_\ast(w^{0,1})\in T^{1,0}_0S\oplus T^{0,1}_0 S.$$ Also, $(p\circ \sigma)_\ast(v)=(\mathrm {id}_S)_\ast(v)=v\in T^{1,0}S$. So $p_\ast(w^{0,1})=0$.
\end{proof}

Later, we shall provide explicit deformation classes for $w^{1,0}$ and $w^{0,1}$, respectively.
\subsection{Deformation theory for a compact Riemann surface}\label{Rie}
In this section, we briefly review some basic deformation theory. Let $f:\mathcal X\to S$ be a family as above with central fiber $X_0:=f^{-1}(0)$, for $0\in S$. Then we have the Kodaira-Spencer map
\begin{align}\label{KS}
\rho_{KS}: T_0S\to H^1(X_0,T_{X_0})
\end{align}

We use the Dolbeault cohomology to represent any class $[\eta]\in H^1(X_0,T_{X_0})$ as follows:
$$H^1(X_0,T_{X_0})=\frac{\mathcal A^{0,1}(T_{X_0})}{\bar\p \mathcal A^{0,0}(T_{X_0})}.$$

Let $\frac{\p}{\p t}\in T_0S$ such that $\rho_{KS}(\frac{\p}{\p t})=[\eta]$.
We may pick some $\eta\in \mathcal A^{0,1}(T_{X_0})$ to represent $[\eta]\in H^1(X_0,T_{X_0})$. 

Now we assume $X_t:=f^{-1}(t)$ is sufficiently close to the central fiber with the complex structure $\mu_t:=t\eta+O(t^2)\in \mathcal A^{0,1}(T_{X_0})$ on $X_t$. We view $X_t$ as the differential manifold $X_0$ endowed with the complex structure $\mu_t$. Then the identity map $$\mathrm{id}:X_0\to (X_0,\mu_t)$$
is a diffeomorphism.

We let $dz$ be a local $(1,0)$ frame on $X_0$ and $d\bar z$ its conjugate. Then $$\om_t:=dz-t\eta(dz)+O(t^2)$$ is a local $(1,0)$ frame for $X_t$, where $\eta(dz)\in\mathcal A^{0,1}(X_0)$ is the contraction.

\subsection{Deformation class of a holomorphic family of Higgs bundles}

In this section, we assume $\sigma: S\to\Mdol$ is a holomorphic section and $\sigma(0)=(E,\bar\p,\theta)\in \mathcal M_{\mathrm{Dol}}(X_0)$. 

We will use the theory of the Atiyah bundle to provide the first-order deformation class of a Higgs bundle in $\Mdol$. One may refer to \cite{Ati} for the definition of the Atiyah class. In \cite[Proposition 4.2.1]{CT}, the author used this theory to describe the deformation class of a triple $(X_0,P,\nabla)$, where $P$ is a $G$-principal bundle on $X_0$ with a holomorphic connection $\nabla$.

For the initial Higgs bundle $(E,\bar\p,\theta)\in \mathcal M_{\mathrm{Dol}}(X_0)$, let $P\stackrel{j}{\to}X_0$ be its frame bundle, which is a $\mathrm{GL}(n,\mathbb C)$-principal bundle. Then we have the following short exact sequence $$0\to T_{P/X}\to T_P\to j^*T_{X_0}\to 0.$$
$G:=\mathrm{GL}(n,\mathbb C)$ acts on both tangent spaces. After taking a quotient by $G$, we have 
\begin{align}\label{at} 0\to \mathrm{ad}P\to \At\to T_{X_0}\to 0.\end{align}
We remark that $\mathrm{ad}P=\mathrm{End}E$ and $\At\cong \mathrm{ad}P\oplus T_{X_0}$ as a smooth vector bundle. 
An Atiyah class is an extension class $$a(E)\in\mathrm{Ext}^1(T_{X_0},\mathrm{End}E)\cong H^1(X,\mathrm{End}E\otimes  K_{X_0}).$$ 

Let $h_0$ be the harmonic metric for $E$ and $\Ch$ be the Chern connection with $\Ch=\Ch^{1,0}+\bar\p$. In our case, by \cite[Proposition 4]{Ati}, we may take 
$$a(E):=-\theta\wedge\theta^{\star_{h_0}}-\theta^{\star_{h_0}}\wedge\theta,$$
which is equal to $F(\Ch)$ by the Hitchin-Yang-Mills equation \eqref{HYM}.
This $a(E)$ defines a complex structure on $\At$ by $$\bar\p_{\At}:=\begin{pmatrix}\bar\p_{\mathrm{End}E} & a(E)\\ 0 &\bar\p_{T_{X_0}}\end{pmatrix}.$$

Let $\theta\lrcorner (\cdot):\mathcal A^{p,q}(T_{X_0})\to \mathcal A^{p,q}(\mathrm{End}E)$ be the contraction map. The operator $\Ch^{1,0}(\theta\lrcorner)$ for any smooth (local) section $\tau$ of $T_{X_0}\otimes\mathcal A^{p,q}(X_0)$ is defined by first performing the contraction $\theta\lrcorner\tau$ and then taking $\Ch^{1,0}(\theta\lrcorner \tau).$

\begin{lemma}\label{2termcx} We have the following 2-term complex 
\begin{align}\label{2-tcx}
\At\xrightarrow{\tad\oplus \Ch^{1,0}(\theta\lrcorner)}\mathrm{End}E\otimes K_{X_0},
\end{align}
which gives the following extension of two deformation complexes
\[\begin{tikzcd}[column sep=2cm]\mathrm{End}E\arrow[r,"\tad"]\arrow[d]&\mathrm{End}E\otimes K_{X_0}\arrow[d,equals]\\
\At \arrow[r,"\tad\oplus \Ch^{1,0}(\theta\lrcorner)"]\arrow[d]& \mathrm{End}E\otimes K_{X_0}\arrow[d]\\T_{X_0}\arrow[r,"0"]&0
\end{tikzcd}\]
\end{lemma}
\begin{proof} As a smooth map, $\tad\oplus \Ch^{1,0}(\theta\lrcorner):\At\cong  \mathrm{End}E\oplus T_{X_0}\to \mathrm{End}E\otimes K_{X_0}$ is well-defined. We verify it is indeed holomorphic, i.e., the following diagram commutes:
\begin{equation}\label{reso}\begin{tikzcd}[column sep=2cm]
\mathcal A^{0,1}(\At)\arrow[r,"\tad\oplus (-\Ch^{1,0}(\theta\lrcorner))"] & \mathcal A^{0,1}(\mathrm{End}E\otimes K_{X_0}) \\ \mathcal A^{0,0}(\At)\arrow[r,"\tad\oplus \Ch^{1,0}(\theta\lrcorner)"]\arrow[u,"\bar\p_{\At}"]&\mathcal A^{0,0}(\mathrm{End}E\otimes K_{X_0})\arrow[u,"\bar\p"]
\end{tikzcd}\end{equation}
For any $(g,\tau)\in \mathcal A^{0,0}(\At)=\mathcal A^{0,0}(\mathrm{End}E)\oplus \mathcal A^{0,0}(T_{X_0}),$ we have 
\begin{align*}\{\tad\oplus (-\Ch^{1,0}(\theta\lrcorner))\}\circ\bar\p_{\At}(g,\tau)=&\tad(\bar\p g)-\Ch^{1,0}(\theta\lrcorner(\bar\p\tau))+\tad\circ a(E)\tau;\\
\bar\p\circ\{\tad\oplus \Ch^{1,0}(\theta\lrcorner)\}(g,\tau)=&\bar\p(\tad g)+\bar\p\Ch^{1,0}(\theta\lrcorner\tau).
\end{align*}
Note that $$\Ch^{1,0}(\theta\lrcorner(\bar\p\tau))+\bar\p\Ch^{1,0}(\theta\lrcorner\tau)=[F(\Ch),\theta\lrcorner\tau]=\tad\circ a(E).$$
We have the desired holomorphic property. The rest is clear by definition.
\end{proof}
Now we present the main result of this section.
\begin{proposition}\label{holtan}Under the assumption $T_0S=H^1(X_0,T_{X_0})$, the hypercohomology group $\mathbb H^1(\At, \tad\oplus \Ch^{1,0}(\theta\lrcorner))$ of \eqref{2-tcx} gives the first order deformation class of $(E,\bar\p,\theta)\in\Mdol$, i.e., \begin{align}\label{Holtan}T_{(E,\bar\p,\theta)}\Mdol=\mathbb H^1(\At, \tad\oplus \Ch^{1,0}(\theta\lrcorner)).\end{align} We also have the exact sequence
\begin{align}\label{shortH}0\to \Hy\to \mathbb H^1(\At, \tad\oplus \Ch^{1,0}(\theta\lrcorner))\stackrel{j_\ast}{\rightarrow}H^1(X_0,T_{X_0})\to 0
.\end{align}
\end{proposition}
\begin{proof}Consider the truncated complex of \eqref{reso}, we have
\begin{align*}\mathcal A^{0,0}(\At)\xrightarrow{D^0}\mathcal A^{0,1}(\At)\oplus \mathcal A^{0,0}(\mathrm{End}E\otimes K_{X_0})\xrightarrow{D^1}\mathcal A^{0,1}(\mathrm{End}E\otimes K_{X_0}),\\
\text{where\ } D^0=(\bar\p_{\At},\tad\oplus \Ch^{1,0}(\theta\lrcorner));\ D^1:=\{\tad\oplus (-\Ch^{1,0}(\theta\lrcorner))\}\oplus\bar\p.
\end{align*}
This truncated complex \textbf{computes} the hypercohomology group in \eqref{Holtan}. So any class in $\mathbb H^1(\At, \tad\oplus \Ch^{1,0}(\theta\lrcorner))$ can be represented by a $D^1$-closed cocycle \begin{align*}(\varphi,\psi,\eta)\in &\Ehol\oplus\Eahol\oplus \mathcal A^{0,1}(T_{X_0})\\&=\mathcal A^{0,0}(\mathrm{End}E\otimes K_{X_0})\oplus\mathcal A^{0,1}(\At),\end{align*}
denoted by $[(\varphi,\psi,\eta)]\in \mathbb H^1(\At, \tad\oplus \Ch^{1,0}(\theta\lrcorner)).$

We give an elementary proof of \eqref{Holtan}. 

For any $t\in S$, let $(E,\bar\p_t,\theta_t)$ be the Higgs bundle given by $\sigma(t)$ over $X_t$. As in the previous section, $\mathrm{id}:X_0\to X_t$ is a diffeomorphism, and we use it to pull back $(E,\bar\p_t,\theta_t)$ to a triple on $X_0$. One tries to compare the pull-back triple $(E,\bar\p_t,\theta_t)$ with $(E,\bar\p,\theta)$ on $X_0$.\\
\textbf{Claim:} There exists $(\varphi,\psi,\eta)\in \Ehol\oplus\Eahol\oplus \mathcal A^{0,1}(T_{X_0})$ such that
\begin{align*}X_t\text{ has } &t\eta+O(t^2)\in \mathcal A^{0,1}(T^{1,0}_{X_0})\text{ as its complex structure};\\
\bar\p_t=&\bar\p+t\eta\circ\Ch^{1,0}-\bar t\bar\eta\circ \bar\p+t\psi+O(|t|^2);\\
\theta_t=&\theta-t\eta(\theta)+t\varphi+O(|t|^2).
\end{align*}

Note that for any $h\in C^\infty(X_0)$, the $\bar\p_t$ operator for $X_t$ is given by
$$\bar\p_t h=\bar\p h+t\eta\circ\p h-\bar t\bar\eta\circ\bar\p h+O(|t|^2).$$

Thus for $\bar\p_t$ of $E$ on $X_t$, we have $\bar\p_t=\bar\p+t\eta\circ\Ch^{(1,0)}-\bar t\bar\eta\circ \bar\p+t\psi+O(|t|^2)$ for some $\psi$ being a smooth (0,1)-form of $\mathrm{End}E$. 
Now we assume $\theta_t=(A_0+tA_1+O(t^2))\omega_t$ where $A_0$ and $A_1$ are smooth sections of $\mathrm{End}E$ and $\om_t$ is a (1,0) local frame of $X_t$ defined in section \ref{Rie}. Then 
$$\theta_t=A_0dz+tA_1dz-t\eta(A_0dz)+O(t^2).$$
Thus we must have $A_0dz=\theta$, and this completes the proof of our claim.

Now we consider the compatible condition $\bar\p_t\theta_t=0$ for the cocycle $(\varphi,\psi,\eta)$ in our claim. For any smooth section $B\omega_t$ of $\mathrm{End}E\otimes K_{X_t}$ and $a\in\mathcal A^{0,0}(E)$, we have
\begin{align*}&\big(\bar\p_t(B\omega_t)\big)(a)=(\bar\p_tB)(a)\omega_t+B(a)\bar\p_t\omega_t\\
=&\{\bar\p_t(B(a))-B(\bar\p_ta)\}\omega_t+(Ba)\bar\p_t\omega_t\\
=&(\bar\p B)\wedge (dz)(a)-t\Ch^{1,0}(\eta(Bdz))(a)+t[\psi,Bdz](a)+O(|t|^2).
\end{align*}Take $B\omega_t=\theta_t$ in the above equality and we get 
\begin{align}\label{dgla}
\bar\p\varphi+[\psi,\theta]=\Ch^{1,0}(\eta(\theta)),
\end{align}
where $\eta:\mathcal A^{p,q}(\mathrm{End}E)\to\mathcal A^{p-1,q+1}(\mathrm{End} E)$ is the contraction. This \eqref{dgla} is just the \textbf{closed condition} $D^1(\varphi,\psi,\eta)=0$ defined in the truncated complex of \eqref{reso}.

Besides, $(g,0)\in \mathcal A^{0,0}(\mathrm{End} E)\oplus\mathcal A^{0,0}(T_{X_0})=\mathcal A^{0,0}(\At)$ defines an exact class$$D^0(g,0)=(\tad(g),\bar\p g,0).$$
This exact term arises from a gauge transform $G:=\mathrm{id}_E+tg+O(t^2)$ of $(E,\bar\p,\theta)$, for any such $g$. (If we consider $(0,\tau)\in\mathcal A^{0,0}(\mathrm{End} E)\oplus\mathcal A^{0,0}(T_{X_0})$ with $\tau\ne0$, the induced exact term $D^0(0,\tau)$ comes from the diffeomorphism of $X_0$ homotopic to the identity map of $X_0$.) This completes the proof of \eqref{Holtan}.

By Lemma \ref{2termcx}, we have the induced long exact sequence of hypercohomology groups. Since $H^0(X_0,T_{X_0})=0$ and one can check that the connecting homomorphism $\delta:H^1(X_0,T_{X_0})\to \mathbb H^2(X_0,(\mathrm{End}E,\mathrm{ad}(\theta)))$ is zero, we have \eqref{shortH}.
\end{proof}

\begin{remark}
1. If $T_0S$ is not equal to $H^1(X_0,T_{X_0})$, we let $\mathcal L:=\rho_{KS}(T_0S)\subset H^1(X_0,T_{X_0})$. Then the deformation class of $(E,\bar\p,\theta)\in\Mdol$ is $$j_\ast^{-1}(\mathcal L)\subset\mathbb H^1(\At, \tad\oplus \Ch^{1,0}(\theta\lrcorner)),$$
where $j_\ast$ is defined in \eqref{shortH}.\\
2. Fix $\frac{\p}{\p t}\in T_0S$. Let 
\begin{align}\label{Aspace}
A_{\p/\p t}:=\{[(\varphi,\psi,\eta)]\in \mathbb H^1(\At, \tad\oplus \Ch^{1,0}(\theta\lrcorner)):\ [\eta]=\rho_{KS}(\frac{\p}{\p t})\}\end{align}
which is equal to 
$\{v\in T^{1,0}_{\sigma(0)}\Mdol:p_\ast(v)=\frac{\p}{\p t}\}.$ For simplicity, we write a class in $A_{\p/\p t}$ as $[(\varphi,\psi)]$ omitting $\eta$. For any holomorphic section $\sigma:S\to\Mdol$, we have $$\sigma_\ast(\p /\p t)=[(\varphi,\psi)]\in A_{\p /\p t}.$$
3. Since $H^2(X_0,\mathrm{End}E)=0$, we may take $\psi=0$ in \eqref{dgla} to get a 'trivial' deformation $(E,\bar\p_t)$ of the holomorphic bundle $(E,\bar\p)$. Then \eqref{dgla} is the equation
$$\bar\p\varphi=\Ch^{1,0}(\eta(\theta)),$$
which coincides with the deformation equation in \cite[Theorem 1.3]{LRY} solving a holomorphic $(1,0)$ section of $(\mathrm{End}E,\bar\p_t)$ for each $t$ with the initial section $\theta\in \Gamma(\mathrm{End}E\otimes K_{X_0})$.

\end{remark}
\subsection{Deformation class of a real analytic family of Higgs bundles}
In this section, we assume $\sigma: S\to\Mdol$ is a real analytic section. We focus on its anti-holomorphic deformation class.

First, we recall the anti-holomorphic involution of Deligne's twistor space $\mathcal{M}_\mathrm{Del}(X_0) \to \mathbb{P}^1$ in \cite{Simp95}. 
Define the involution map $\sigma': \mathcal{M}_\mathrm{Del}(X_0) \to \mathcal{M}_\mathrm{Del}(X_0)$ by 
\[
(E,\lambda\Ch^{1,0}+\bar\p+\theta+\lambda\theta^{\star_{h_0}},\lambda) \mapsto (\overline E^\vee,-\bar\lambda^{-1}\Ch^{1,0}+\bar\p+\theta-\bar\lambda^{-1}\theta^{\star_{h_0}},-\bar\lambda^{-1}).
\]
We restrict this involution to $\mathcal M_{\mathrm{Dol}}(X _0)\to\{0\}$, 
which gives an anti-holomorphic homeomorphism
\begin{align*}
    \sigma':\mathcal M_{\mathrm{Dol}}(X_0) & \to \mathcal M_{\mathrm{Dol}}(\overline {X_0})\\
    (E,\bar\p,\theta) & \mapsto (\overline E^\vee,\Ch^{1,0},\theta^{\star_{h_0}})
\end{align*}
where $\overline{E}^\vee$ is the vector bundle whose transition functions are inverse to the conjugate transpose of transition functions of $E$, and $\overline{X_0}$ is the Riemann surface diffeomorphic to $X_0$ with the conjugate complex structure.

This gives a $\mathbb C$-linear isomorphism between $T^{0,1}_{(E,\bar{\partial}, \theta)} \mathcal M_{\Dol}(X_0)$ and $T^{1,0}_{(\overline E^\vee,\Ch^{1,0},\theta^{\star_{h_0}})} \mathcal M_{\Dol}(\overline{X_0})$, i.e.
\[
\sigma'_{*,(E,\bar{\partial}, \theta)} : \overline \Hy \stackrel{\simeq}{\rightarrow} \mathbb H^1(\overline{X_0}, (\End \overline E^\vee, \mathrm{ad}(\theta^{\star_{h_0}}))).
\]

We want to use the Dolbeault resolution to compute $\overline\Hy$. So we view $\overline E^\vee\cong E$ by the harmonic metric $h_0$, and then $\mathrm{End}(\overline E^\vee)\cong\mathrm{End}E$. In addition, we clearly have $\Omega_{\overline{X_0}}^{1,0}=\Omega_{X_0}^{0,1}$ and $\Omega_{\overline{X_0}}^{0,1}=\Omega_{X_0}^{1,0}$. Hence, just as in section \ref{htdol}, we also have the following Dolbeault resolution:
\[
\begin{tikzcd}
C^{1,0}:=\mathcal A^{1,0}(\End E  ) \arrow[r, "{\mathrm{ad}( \theta^{\star_{h_0}})}"]                                                   & C^{1,1}:={\mathcal A^{1,0}(\End E \otimes \Omega_{X_0}^{0,1})} \\
C^{0,0}:=\mathcal A^{0,0} (\End E) \arrow[r, "{\mathrm{ad}( \theta^{\star_{h_0}})}"] \arrow[u, "\Ch^{1,0}"] & C^{0,1}:=\mathcal A^{0,0} (\End E \otimes \Omega_{X_0}^{0,1}) \arrow[u, "\Ch^{1,0}"]              
\end{tikzcd}
\]
which gives the following truncated complex  
\[C^{0,0} \overset{d^{0c}}{\longrightarrow} C^{1,0} \oplus C^{0,1} \overset{d^{1c}}{\longrightarrow}  C^{1,1}\]  
where 
\begin{gather*}
    d^{0c}(g) = (\Ch^{1,0} g, [\theta^{\star_{h_0}},g]) \quad \text{for } g \in C^{0,0},\\
    d^{1c}(\varphi, \psi) =  \Ch^{1,0} \psi +[\varphi,\theta^{\star_{h_0}}] \quad \text{for } (\varphi, \psi) \in C^{0,1} \oplus C^{1,0}.
\end{gather*}
Hence
$$T^{0,1}_{(E, \overline{\partial}, \theta)} \mathcal{M}_{\mathrm{Dol}}(X_0)=\overline\Hy=\frac{\mathrm{Ker\ }d^{1c}}{\mathrm{Im\ }d^{0c}}.$$
And we have the following complex conjugate operator 
\begin{align*}
    \mathrm{bar} : \Hy & \rightarrow \overline\Hy\\
    [(\varphi,\psi)] & \mapsto [(\psi^{\star_{h_0}}, -\varphi^{\star_{h_0}})].
\end{align*}

We shall describe how a class $[(\varphi,\psi)]\in\overline\Hy$ deforms $(E,\bar\p,\theta)$ in an anti-holomorphic manner. For this, we need to define the harmonic 1-form.

\begin{definition}\label{harcl}
    Suppose $(\varphi,\psi) \in C^{1,0}\oplus C^{0,1}$, we say $(\varphi,\psi)$ is harmonic if $\bar\p\varphi + [\psi,\theta] =0$ and $ \Ch^{1,0}\psi + [\varphi,\theta^{\star_{h_0}}] = 0$.
\end{definition} 

Then Hitchin \cite{Hit} proved that 

\begin{proposition}
    For each class in $\Hy$ (or in $\overline{\Hy}$), there is a unique harmonic representative $(\varphi , \psi)$ of this class as defined in Definition \ref{harcl}.
\end{proposition}

Any harmonic class $[(\varphi,\psi)]\in\overline\Hy$ defines a (first-order) anti-holomorphic deformation of $(E,\bar\p, \theta)$ by
$$\bar\p_t=\bar\p+\bar t\psi+O(|t|^2);\ \theta_t=\theta+\bar t\varphi+O(|t|^2).$$

Combining this with Proposition \ref{holtan} and Lemma \ref{realsec}, we have
\begin{proposition}\label{aholtan}Let $\sigma:S\to\Mdol$ be a real analytic section and $\frac{\p}{\p t}\in T_0S$. Then there exists a class $[(\varphi_1,\psi_1)]\in A_{\p/\p t}$ (defined in \eqref{Aspace}) and a harmonic class $[(\varphi_2,\psi_2)]\in\overline\Hy$ such that $\sigma_\ast(\frac{\p}{\p t})$ has $[(\varphi_1,\psi_1)]\in A_{\p/\p t}$ as its (1,0) part and has $[(\varphi_2,\psi_2)]$ as its (0,1) part. 
\end{proposition}

\section{Isomonodromic deformation of a Higgs bundle}

Let $\sigma: S\to\Mdol$ be the isomonodromic deformation of $(E,\bar\p,\theta)$ on $X_0$ defined in section \ref{Isomd}, which is a real analytic section. In this section, we shall study its first order deformation classes and prove Theorem \ref{thm-A}. The main technique is to relate every deformation class to the first-order deformation of the harmonic metric.

Let $(\mathcal V,D)$ be the flat rank $n$ complex vector bundle over $X_0$ corresponding to $(E,\bar\p,\theta)$. (If we forget the holomorphic structure, $\mathcal V=E$ over $X_0$.) Let $h_0$ be the harmonic metric of $(\mathcal V,D)$ over $X_0$. Note that  $(\mathcal V,D)$ over $X_t$ is the isomonodromy deformation of the initial flat bundle $(\mathcal V,D)$ over $X_0$. Now we let $h$ be the harmonic metric of $(\mathcal V,D)$ over $X_t$, which satisfies the following equation
\begin{align}\label{Hmetric}
    D_h^{\star_h}\psi_h=0,
\end{align}
where $D=D_h+\psi_h$ is the decomposition of $D$ into unitary and Hermitian parts, and $\star_h$ is the Hodge star operator with respect to $h$. 

\begin{lem}\label{Taylor}
    The harmonic metric $h$ has the following Taylor expansion with $g\in\mathcal A^0(\mathrm{End}(E))$ (after pulling back $(\mathcal V,h)$ to $X_0$ by $\mathrm{id}:X_0\to (X_0,\mu_t)=X_t$)
    \begin{align*}
        h=h_0+th_0g+\bar th_0g^{\star_{h_0}}+O(|t|^2).
    \end{align*}
\end{lem}
\begin{proof}
    Since $h$ is a Hermitian metric, one has $h=h_0+th_1+\bar t\bar h_1^t+O(|t|^2))$. Then we may assume $h_1=h_0g$ for some $g$ being a section of End$(E)$. Then $\bar h_1^t=\bar g^t h_0=h_0 g^{\star_{h_0}}.$ 
\end{proof}
We also view $h$ as a harmonic map on $\widetilde X$, and by \cite[Lemma 2.16]{Li} we have the following formula for $\psi_h$
\begin{align}\label{psi}
    \psi_h=-\frac{1}{2}h^{-1}Dh.
\end{align}

Combining \eqref{Hmetric} and \eqref{psi}, we will derive a PDE for $g$ later in this section. This $g$ determines how the harmonic metric deforms. Therefore, $g$ determines the first-order deformation class of the isomonodromic deformation of a Higgs bundle as below.

\begin{prop}\label{thetadbar}The first order isomonodromic  deformation can be explicitly expressed as follows
\begin{align}
  \bar\p_t
  &=\bar\p
    +t\left(\eta\circ D_{h_0}^{1,0}
             -\frac12[\theta^{\star_{h_0}},g]\right)
    +\bar t\left(-\bar\eta\circ\bar\p
             -\frac12[\theta^{\star_{h_0}},g^{\star_{h_0}}]\right)
    +O(|t|^2),\label{eq:dbar-short}\\
  \theta_t
  &=\theta-t(\eta(\theta))
    -\frac t2D_{h_0}^{1,0}g
    +\bar t\left(\bar\eta(\theta^{\star_{h_0}})
             -\frac12D_{h_0}^{1,0}g^{\star_{h_0}}\right)
    +O(|t|^2).\label{eq:theta-short}
\end{align}
\end{prop}
\begin{proof}Recall on $X_0$, the flat smooth connection $D$ has the following decomposition 
$$D=D_{h_0}+\theta+\theta^{\star_{h_0}}$$ with $D_{h_0}$ being the Chern connection of $(E,h_0)$ on $X_0$.
By Lemma~\ref{Taylor} and \eqref{psi}, we have
\begin{align*}\Psi_h=&\Psi_{h_0}+[\Psi_{h_0},tg+\bar tg^{\star_{h_0}}]-\frac12 D(tg+\bar tg^{\star_{h_0}})+O(|t|^2)\\
=&\Psi_{h_0}+\frac12[\Psi_{h_0},tg+\bar tg^{\star_{h_0}}]-\frac12 D_{h_0}(tg+\bar tg^{\star_{h_0}})+O(|t|^2).
\end{align*}
By comparing its $(1,0)$ and $(0,1)$ parts on $X_t$, we get (after modulo $|t|^2$)
\begin{equation}\label{eq_thetat_thetatstar}
\begin{aligned}\theta_t=&\theta+t(-\eta(\theta)+\frac12[\theta,g]-\cc D_{h_0}^{1,0} g)+\bar t(\bar \eta(\theta^{\star_{h_0}})+\frac12[\theta,\bg]-\cc D_{h_0}^{1,0} \bg),\\
\theta_t^{\star_h}=& \theta^{\star_{h_0}}+t(\eta(\theta)+\frac12[\theta^{\star_{h_0}},g]-\cc \bar\p g)+\bar t(-\bar \eta(\theta^{\star_{h_0}})+\frac12[\theta^{\star_{h_0}},\bg]-\cc \bar\p \bg).
\end{aligned}\end{equation}
The $(0,1)$-part of $D-\theta_t^{\star_h}$ on $X_t$ gives $\bar\p_t$. Thus we can use \eqref{eq_thetat_thetatstar} to compute $\bar\p_t$ and $\theta_t$. After taking the gauge transformation $\mathscr U:=\operatorname{id}-\frac12(tg+\bar t g^{\star_{h_0}})$, we get \eqref{eq:dbar-short} and \eqref{eq:theta-short}.
\end{proof}

Now we determine the PDE that $g$ satisfies. This linear non-homogeneous PDE can be viewed as the linearization of the non-linear equation \eqref{Hmetric} by using the isomonodromic deformation to 'eliminate' the non-linear term. 
 
\begin{prop}\label{pde}
    Two smooth endomorphisms of $E$, $g$ and $\bg$, satisfy the following equations:
    \begin{equation}\label{equ_g}
    \begin{aligned}
        \bar\p\Ch^{1,0} g&=-2\Ch^{1,0}(\eta(\theta))-[\theta,[\theta^{\star_{h_0}},g]];\\
        \bar\p\Ch^{1,0} \bg&=2\bar\p(\bar \eta(\theta^{\star_{h_0}}))-[\theta,[\theta^{\star_{h_0}},\bg]].
    \end{aligned}\end{equation}
\end{prop}
\begin{proof}Note that $X_t$ is a family of compact Riemann surfaces, the equation \eqref{Hmetric} is equivalent to the equation $\bar\p_t\theta_t=0$. Applying \eqref{dgla} to the $t$-coefficient of \eqref{eq:dbar-short} and \eqref{eq:theta-short}, we get the first equation of \eqref{equ_g}. A similar argument applies to $\bg$. 
\end{proof}

Now we apply the deformation theory established in section \ref{DT}, including Proposition \ref{holtan} and Proposition \ref{aholtan} to prove the theorem on the first order deformation classes.

\begin{proof}[Proof of Theorem~\ref{thm-A}](i) By Proposition \ref{thetadbar} and Proposition \ref{holtan}, $[(-\cc\Ch^{1,0}g,-\cc [\theta^{\star_{h_0}},g],\eta)]$ is simply the holomorphic deformation class.\\
(ii) Clearly, by Proposition \ref{aholtan} and the above, $$[(\bar\eta(\theta^{\star_{h_0}})-\cc\Ch^{1,0}\bg,-\cc [\theta^{\star_{h_0}},\bg])]\in T_{\sigma(0)}^{0,1}\mathcal M_{\mathrm{Dol}}(X_0)$$ is the deformation class, which is a harmonic representative. This class is just $[(\bar\eta(\theta^{\star_{h_0}}),0)]\in\overline\Hy$ after modulo the exact term.
\end{proof}

So the class $[(0,\eta(\theta))]\in \Hy$ given by the complex conjugate of $\Phi_{KS}^{0,1}(\p/\p t)$ as above, is the obstruction for the isomonodromic section $\sigma:S\to\Mdol$ being holomorphic along $\frac{\p}{\p t}$. We have the following example where $\Phi_{KS}^{0,1}$ is not zero:

\begin{exa}\label{Hiteg} Recall Hitchin's uniformization Higgs bundle $(E:=K_{X_0}^{\cc}\oplus K_{X_0}^{-\cc},\theta=\begin{pmatrix}0&1 \\0 & 0\\\end{pmatrix})$ over $X_0$ defined in \cite[Corollary 4.23]{Hit}. Clearly, $$\theta_\ast: H^1(X_0,T_{X_0})\to \Hy$$
is injective. For other Higgs bundle $(E,\bar\p,\theta)$, we shall study whether $\theta_\ast$ is injective (or non-zero) in Section \ref{sec5-1}.
\end{exa}

\begin{rmk}We focus on the rank 1 case, i.e., $\mathcal M_{\mathrm{Dol}}(\mathcal X/S,\mathbb C^\ast)=f_\ast\Omega^1_{\mathcal X/S}\times_S \mathrm{Jac}(\mathcal X/S)$. Let $\varphi$ be the usual Higgs field of the weight-one variation of Hodge structure
\[
  R^1f_*\mathbb C\otimes\mathcal O_S
  \cong f_*\Omega^1_{\mathcal X/S}\oplus R^1f_*\mathcal O_\mathcal X.
\]
Let $[\eta]$ be a Kodaira-Spencer class in the image of $\rho_{\mathrm{KS}}$. For example, $[\eta]=\rho_{\mathrm{KS}}(\p/{\p t})$. The expression $\theta_*([\eta])$ defined in \eqref{Theta_ks} is exactly $\varphi_{\p/\p t}(\theta)$. 
Thus Theorem~\ref{thm-A}(ii) recovers the usual Higgs field in rank $1$ case.\end{rmk}

\section{Isomonodromic deformation of a graded Higgs bundle}
Let $(E,\bar\p,\theta)$ be a \textbf{graded Higgs bundle} on $X_0$, i.e., $(E=\bigoplus\limits_{p=0}^kE^{p,k-p},\theta=\bigoplus\limits_{p=1}^k\theta^{p,k-p})$ such that $\theta^{p,k-p}:E^{p,k-p}\to E^{p-1,k-p+1}\otimes  K_{X_0}$. Let $\sigma: S\to\Mdol$ be the isomonodromic deformation of $(E,\bar\p,\theta)$ defined in section \ref{Isomd}. Then we have a family of Higgs bundles $(E,\bar\p_t,\theta_t)$ on $X_t$. 

We aim to show that for sufficiently small $t$, the Higgs field $\theta_t$ is \textbf{non-nilpotent} along the tangent direction $\frac{\p}{\p t}\in T_0S$ if $\Phi^{0,1}_{KS}$ is non-zero along this direction.

First, we have a good expression for $g$ in Proposition \ref{pde}:

\begin{lem}\label{g} The endomorphism $g$ has the decomposition $g=\bigoplus\limits_{p=1}^kg^{p,k-p}$ where $g^{p,k-p}:E^{p,k-p}\to E^{p-1,k-p+1}$ is the composition of $g|_{E^{p,k-p}}$ with $E\to E^{p-1,k-p+1}$.\end{lem}
\begin{proof}By a direct check and the fact that $h_0=(\bigoplus h_0|_{E^{p,k-p}})$ is block-wise diagonal (proved in \cite[Lemma 3.5]{Li}), the PDE \eqref{equ_g} for $g$ in Proposition \ref{pde} is linear for $g|_{E^{p,k-p}}$ composed with $E\to E^{q,k-q}$ for any $q\ne p-1$. By the uniqueness of the solution of the harmonic metric equation \eqref{Hmetric} (up to a parallel section), we obtain the desired result.
\end{proof}

\begin{thm}\label{nonnil} Assume (a) $\sigma(0):=(E,\bar\p,\theta)$ is a graded Higgs bundle on $X_0$, and \\(b) the non-abelian Kodaira-Spencer map $\Phi^{0,1}_{KS}$ is non-zero at $\sigma(0)$ for some $\frac{\p}{\p t}\in T_0S$.\\
Then the isomonodromic deformation of $(E,\bar\p,\theta)$ along the direction $\frac{\p}{\p t}\in T_0S$ is non-nilpotent.
\end{thm}
\begin{proof}Let $[\eta]:=\rho_{KS}(\frac{\p}{\p t})$. We choose a representative $\eta\in \mathcal A^{0,1}(T_{X_0})$ as usual. Then we have $[(0,\eta(\theta))]\ne0\in\Hy$ by Lemma \ref{thetaast}.

Note that $\bar \eta(\theta^{\star_{h_0}})$ consists only of the following smooth sections for $p=0,1,\cdots,k-1$
$$\bar \eta((\theta^{p+1,k-p-1})^{\star_{h_0}}):E^{p,k-p}\to E^{p+1,k-p-1}\otimes K_{X_0},$$
since the harmonic metric $h_0$ is block-wise diagonal.
By Lemma \ref{g}, one has 
\begin{itemize}
\item $\eta(\theta)$ and $\Ch^{1,0}g$ map $E^{p,k-p}$ to $\mathcal A^1(E^{p-1,k-p+1})$; 
\item $[\theta,g]$ maps $E^{p,k-p}$ to $E^{p-2,k-p+2}\otimes \Omega^{1,0}_{X_0}$; 
\item $[\theta,\bg]$ maps $E^{p,k-p}$ to $E^{p,k-p}\otimes \Omega^{1,0}_{X_0}$; 
\item $\bar \eta(\theta^{\star_{h_0}})-\cc\Ch^{1,0}\bg$ maps $E^{p,k-p}$ to $E^{p+1,k-p-1}\otimes \Omega^{1,0}_{X_0}$.
\end{itemize}
Then we can compute $\mathrm{tr}(\theta_t^2)$:
\begin{align}\label{eq_tr2}\mathrm{tr}(\theta_t^2)=\bar t\cdot \mathrm{tr}(\theta(2\bar \eta(\theta^{\star_{h_0}})-\Ch^{1,0}\bg))+O(|t|^2).\end{align}
It suffices to prove $\mathrm{tr}(\theta(2\bar \eta(\theta^{\star_{h_0}})-\Ch^{1,0}\bg))\not\equiv 0.$
By the PDE \eqref{equ_g} for $\bg$ in Proposition \ref{pde}, we set
$$\phi:=2\bar \eta(\theta^{\star_{h_0}})-\Ch^{1,0}\bg\in \mathcal A^{1,0}(\mathrm{End}E)\quad \text{with}\quad \bar\p\phi=[\theta,[\theta^{\star_{h_0}},\bg]].$$
For any $\varphi_1,\varphi_2\in \Ehol$, we have their inner product$$ i\int_{X_0}\mathrm{tr}(\varphi_1\wedge\varphi_2^{\star_{h_0}})\overset{\text{if }\varphi_1=\varphi_2}{=}\|\mathrm{tr}(\varphi_1\wedge\varphi_1^{\star_{h_0}})\|_{L^1}.$$
Thus \begin{align*}&\|4\mathrm{tr}(\bar\eta(\theta^{\star_{h_0}})\wedge(\bar\eta(\theta^{\star_{h_0}}))^{\star_{h_0}})\|_{L^1}\\=&\|\mathrm{tr}(\Ch^{1,0}\bg\wedge(\Ch^{1,0}\bg)^{\star_{h_0}})\|_{L^1}+\|\mathrm{tr}(\phi\wedge\phi^{\star_{h_0}})\|_{L^1}+2\mathrm{Re}(i\int_{X_0}\mathrm{tr} (\phi\wedge (\Ch^{1,0}\bg)^{\star_{h_0}})).
\end{align*}
By the K\"ahler identity for $\Ch^{1,0}$ (see \cite{Simp92} and \cite[Remark 9.2]{Gui}), we have
\begin{align*}&2\mathrm{Re}(i\int_{X_0}\mathrm{tr} (\phi\wedge (\Ch^{1,0}\bg)^{\star_{h_0}}))=2\mathrm{Re}(\int_{X_0}\mathrm{tr} (i\Lambda_{\omega_0}\bar\p\phi\cdot(\bg)^{\star_{h_0}}))
\\=&2\mathrm{Re}(\int_{X_0}i\Lambda_{\omega_0}\mathrm{tr} ([\theta,[\theta^{\star_{h_0}},\bg]]\cdot g)
=2\|\mathrm{tr} ([\theta,g]\wedge[\theta,g]^{\star_{h_0}})\|_{L^1}.
\end{align*}

If $\phi\ne0$ or $[\theta,g]\ne0$, we have \[\|4\mathrm{tr}(\bar\eta(\theta^{\star_{h_0}})\wedge(\bar\eta(\theta^{\star_{h_0}}))^{\star_{h_0}})\|_{L^1}>\|\mathrm{tr}(\Ch^{1,0}\bg\wedge(\Ch^{1,0}\bg)^{\star_{h_0}})\|_{L^1}.\] 
 Thus by Cauchy's inequality
\begin{align*}\|4\mathrm{tr}(\bar\eta(\theta^{\star_{h_0}})\wedge(\bar\eta(\theta^{\star_{h_0}}))^{\star_{h_0}})\|_{L^1}>\|2\mathrm{tr}(\Ch^{1,0}\bg\wedge(\bar\eta(\theta^{\star_{h_0}}))^{\star_{h_0}})\|_{L^1}.
\end{align*}
This implies the non vanishing of the right hand side of \eqref{eq_tr2}.

Now we focus on the case that $\phi\equiv 0$ and $[\theta,g]\equiv0$. But this contradicts the assumption that $[(0,\eta(\theta))]\in \Hy$ is not zero.
\end{proof}

\begin{rmk} By Example \ref{Hiteg}, if $\rho_{KS}$ is injective, Theorem \ref{nonnil} applies to the isomonodromic deformation of Hitchin's uniformization Higgs bundle $(E:=K_{X_0}^{\cc}\oplus K_{X_0}^{-\cc},\theta=\begin{pmatrix}0&1 \\0 & 0\\\end{pmatrix})$ over $X_0$. In this case, its (infinitesimal) isomonodromically deformed Higgs bundles are non-nilpotent. Using \cite[Corollary 4.23]{Hit}, one can prove if $\rho_{KS}$ is injective, then any (not just infinitesimal) isomonodromic deformation of Hitchin's uniformization Higgs bundle is not nilpotent.
\end{rmk}

\section{Non-existence of holomorphic isomonodromic deformation of Higgs bundles}\label{sec5-1}
In this section, let $(E,\bar\p,\theta)$ be a $\mathrm{GL}(n,\mathbb C)$-polystable Higgs bundle over a compact Riemann surface $X_0$ of genus $g\geq 2$. We focus on the induced map defined in \eqref{Theta_ks}
$$\theta_\ast:H^1(X_0,T_{X_0})\to\Hy.$$
We give several lemmas before proving Theorem \ref{main B} and Theorem \ref{main C}.
\begin{lem}\label{redsl}If $\mathrm{tr}(\theta)\ne0$ in $H^0(K_{X_0})$, then $\theta_\ast$ is a non-zero map.
\end{lem}
\begin{proof} We claim that for a non-zero $\om:=\mathrm{tr}(\theta)$ in $H^0(K_{X_0})$, there exists $[\eta]\in H^1(T_{X_0})$ such that $[\eta(\om)]\ne0\in H^1(\mathcal O_{X_0}).$ 

By  Serre's duality $H^1(\mathcal O_{X_0})=H^0(K_{X_0})^\vee$, it suffices to find some $\beta\in H^0(K_{X_0})$ such that $[\eta(\om\cdot \beta)]\in H^1(K_{X_0})$ is non-zero. Note that if $\om$ and $\beta$ are both non-zero in $H^0(K_{X_0})$, their multiplication $\om\cdot \beta\in H^0(K_{X_0}^{\otimes 2})$ is non-zero. Then we have some $[\eta]\in H^1(T_{X_0})$ such that $[\eta(\om\cdot \beta)]\in H^1(K_{X_0})$ is non-zero because\begin{align}\label{pefp}H^1(T_{X_0})\times H^0(K_{X_0}^{\otimes2})\to H^1(K_{X_0})=\mathbb C
\end{align} is a perfect pair.

If $\theta_\ast$ is the zero map, then for any $[\eta]\in H^1(T_{X_0})$, by Lemma \ref{thetaast}, we have $[(0,\eta(\theta))]\in\Hy$ is zero. Then for any $[\eta]$, there exists a smooth section $g\in\mathcal A^0(\End E)$ such that $$[\theta,g]=0\text{ and }\bar\p g=\eta(\theta),$$
and hence $[\eta\lrcorner \mathrm{tr}(\theta)]=[\mathrm{tr}(\eta(\theta))]=[\bar\p(\mathrm{tr}(g))]=0\in H^1(\mathcal O_{X_0})$, which is a contradiction.
\end{proof}
By lemma \ref{redsl}, if the isomonodromic deformation of a $\mathrm{GL}(n,\mathbb C)$-Higgs bundle $(E,\bar\p,\theta)$ is holomorphic over $\mathcal T_g$, then this Higgs bundle must be a $\SL$-Higgs bundle.
\begin{lem}\label{H1Ker} Recall the sheaf morphism $\tad:\End E\to \End E\otimes K_{X_0}$ and we have the subsheaf $\Ker(\tad)\subset \End E$. Then the induced map $\theta_\ast$ factors through $H^1(\Ker(\tad))$. And we have the following commutative diagram
\[
\begin{tikzcd}
H^1(T_{X_0}) \arrow[r, "\theta_\ast"]  \arrow[rd,"\theta_\ast"']                                                 & \Hy \\
& H^1(\Ker(\tad)) \arrow[u, hook]              
\end{tikzcd}
\]
\end{lem}
\begin{proof}We have the following diagram of two short exact sequences
\[
\begin{tikzcd}[column sep=small]
0 \arrow[r, ]&\Ker(\tad) \arrow[r, ] \arrow[d, "\tad"]& \End E \arrow[r,] \arrow[d, "\tad"]& Q \arrow[r, ] \arrow[d, "\tad"]& 0 \\
0 \arrow[r, ]& 0 \arrow[r, ]&\End E\otimes K_{X_0}\arrow[r, equal]
 & \End E\otimes K_{X_0} \arrow[r, ]&0
\end{tikzcd}
\]
where $Q$ is the quotient sheaf. Since, $\tad:Q\to\End E\otimes K_{X_0}$ is injective, one can check by definition that $$\mathbb H^0(Q\stackrel{\tad}{\longrightarrow} \End E\otimes K_{X_0})=0.$$ Thus we have $H^1(\Ker(\tad))\hookrightarrow\Hy$. Since $\theta: T_{X_0}\to \End E$ factors through $\Ker(\tad)$, we know that $\theta_\ast$ factors through $H^1(\Ker(\tad))$.
\end{proof}

\begin{lem}\label{limcase} Let $(E,\bar\p,\theta)$ be a nilpotent non-unitary Higgs bundle. And $(E',\bar\p',\theta')$ is defined to be  $\lim\limits_{\lambda\in\mathbb C^\ast,\lambda\to \infty}\lambda\cdot(E,\bar\p,\theta)$. Then $(E',\bar\p',\theta')$ is the associated graded Higgs bundle of $(E,\bar\p,\theta)$ and $\theta'\ne0$. If moreover $\theta'_{\ast}\ne0$, then $\theta_\ast\ne0$.
\end{lem}
\begin{proof}The first claim is proved in \cite[Proposition 3.11]{THit}. For the second claim, note that $$\mathbf M:=\{(E,\bar\p,\theta)\in \mathcal M_{\mathrm{Dol}}(X_0):\theta_\ast=0\}$$
is a closed analytic subvariety of $\mathcal M_{\mathrm{Dol}}(X_0)$ and $\mathbf M$ is invariant under the $\mathbb C^\ast$-action (also observed in \cite[Corollary 1.6]{CTW}). Thus $\theta'_{\ast}\ne0$ implies that $\theta_\ast\ne0.$
\end{proof}
\begin{lem}\label{crigra}Let $(E,\bar\p,\theta)$ be a graded Higgs bundle, i.e. $(E=\bigoplus\limits_{p=0}^kE^{p,k-p},\theta=\bigoplus\limits_{p=1}^k\theta^{p,k-p})$ such that $\theta^{p,k-p}:E^{p,k-p}\to E^{p-1,k-p+1}\otimes K_{X_0}$. If $\theta_\ast\equiv 0$, then \begin{align*}\theta_\ast^{p,k-p}:H^1(T_{X_0})&\to H^1(E^{p,k-p\vee}\otimes E^{p-1,k-p+1})\\ [\eta]&\mapsto [\eta(\theta^{p,k-p})]\end{align*}
are zero maps for $p=1,\cdots, k$.
\end{lem}
\begin{proof} For any $[\eta]\in H^1(T_{X_0})$, if $[(0,\eta(\theta))]\in\Hy$ is a zero class. Then there exists $g\in \mathcal A^0(\End E)$ such that $$[\theta,g]=0\text{ and }\bar\p g=\eta(\theta).$$
Let $g^{p,k-p}:E^{p,k-p}\to E^{p-1,k-p+1}$  be the composition of $g|_{E^{p,k-p}}$ with $E\to E^{p-1,k-p+1}$. Then $[\eta(\theta^{p,k-p})]=[\bar\p g^{p,k-p}]=0\in H^1(E^{p,k-p\vee}\otimes E^{p-1,k-p+1})$.
\end{proof}
Let $\lambda$ be the coordinate of the space $\mathrm{Tot}(K_{X_0}).$ The spectral curve of $(E,\bar\p,\theta)$ is defined by $$\Sigma_\theta:=\{\lambda\in \mathrm{Tot}(K_{X_0}):\det(\lambda\cdot\mathrm{id}_E-\theta)=0\}.$$ When the spectral curve is reduced, by the BNR correspondence in \cite{BNR}, we have a torsion free rank 1 coherent sheaf $\mathcal F$ on $\Sigma_\theta$ such that \begin{align}\label{BNRc}(\pi_\ast\mathcal F,\pi_\ast (\lambda|_{\Sigma_\theta}))=(E,\theta).\end{align}
Besides, the BNR correspondence gives the following result on the Hitchin fibration (see \cite{BNR,Daniel})

\begin{lem}\label{Hitfib}
Let $(E,\bar\p,\theta)$ be a Higgs bundle with reduced (but could be reducible) spectral curve $\Sigma_\theta$ and eigenforms $(a_1,\cdots,a_n)\in \bigoplus\limits_{i=1}^n H^0(K_{X_0}^{\otimes i})$.
Then the fiber $h^{-1}(a_1,\cdots,a_n)$ of the Hitchin fibration $ h:\mathcal M_{\mathrm{Dol}}(X_0)\to \bigoplus\limits_{i=1}^n H^0(K_{X_0}^{\otimes i})
$ is the compactified generalized Jacobian of $\Sigma_\theta$, whose tangent space at $\mathcal F$ (defined in \eqref{BNRc}) is $\mathrm{Ext^1}(\mathcal F,\mathcal F)$. Since this compactified generalized Jacobian is a subvariety of the Dolbeault moduli space, we have $$\mathrm{Ext^1}(\mathcal F,\mathcal F)\hookrightarrow\Hy.$$
\end{lem}
\subsection{A generic Higgs bundle}\label{sec_generic}
By Theorem \ref{thm-A} (ii), we can prove Theorem \ref{main B} by proving
\begin{proposition}\label{inje}Let $n\geq 2$. Let $(E,\bar\p,\theta)$ be a Higgs bundle over a compact Riemann surface $X_0$ of genus $g\geq 2$ and its spectral curve $\Sigma_{\theta}\subset \mathrm{Tot}(K_{X_0})$ is smooth. The induced map $\theta_\ast:H^1(X_0,T_{X_0})\to\Hy$ is injective. 
\end{proposition}
\begin{proof}Firstly, $\theta_\ast$ factors through $H^1(\Ker (\tad))$ by Lemma \ref{H1Ker}. Note that by \cite[Proposition 2.4]{moch24}, we have
$$\Ker(\tad)\cong\pi_\ast\mathcal O_{\Sigma_\theta},$$ where $\pi:\Sigma_\theta\to X_0$ is the projection map. When $\Sigma_\theta$ is smooth, by \cite[Page 173]{BNR}, we have $$\pi_\ast\mathcal O_{\Sigma_\theta}=\mathcal O_{X_0}\oplus T_{X_0}\oplus\cdots\oplus T_{X_0}^{\otimes (n-1)}.$$ Thus $H^1(\Ker (\tad))=\bigoplus\limits_{i=0}^{n-1}H^1(T_{X_0}^{\otimes i})$. Hence $\theta_\ast$ is injective.
\end{proof}
\subsection{Rank 2 case}\label{sec_rk2}
By Theorem \ref{thm-A} (ii), we can prove Theorem \ref{main C} (i) by proving
\begin{proposition}\label{rk2}Let $(E,\bar\p,\theta)$ be a polystable rank 2 Higgs bundle over a compact Riemann surface $X_0$ of genus $g\geq 2$. If $\theta\ne0$, the induced map $\theta_\ast:H^1(X_0,T_{X_0})\to\Hy$ is always non-zero. 
\end{proposition}
\begin{proof}By Lemma \ref{redsl}, we may assume $\mathrm{tr}(\theta)$ is zero.

{\bf Case I:} $\mathrm{tr}(\theta^2)\ne0\in H^0(K_{X_0}^{\otimes2})$. We clearly have some $[\eta]\in H^1(T_{X_0})$ such that $[\eta\lrcorner \mathrm{tr}(\theta^2)]\ne0\in H^1(K_{X_0})$ by \eqref{pefp}.  
 Thus we can prove  $[(0,\eta(\theta))]\in\Hy$ is non-zero. Otherwise, there exists a smooth section $g\in\mathcal A^0(\End E)$ such that $$[\theta,g]=0\text{ and }\bar\p g=\eta(\theta),$$
and hence $[\eta\lrcorner \mathrm{tr}(\theta^2)]=[\mathrm{tr}(\theta\cdot\eta(\theta))]=[\bar\p(\mathrm{tr}(\theta\cdot g))]=0\in H^1(K_{X_0})$, which is a contradiction.

{\bf Case II:} $\det\theta=-\cc \mathrm{tr}(\theta^2)=0$, i.e. the nilpotent case. By Lemma \ref{limcase}, we just need to prove this proposition for non-unitary graded Higgs bundles. Then there exists a holomorphic line bundle $L$ such that $(E=L\oplus L^{-1},\theta=\begin{pmatrix}0&\alpha\\0&0\end{pmatrix})$ with $-(g-1)\leq\deg L\leq 0$ and  $\alpha\ne0\in H^0(L^{\otimes 2}\otimes K_{X_0})$. 

By Lemma \ref{crigra}, we just need to prove that $\alpha_{\ast}:H^1(T_{X_0})\to H^1(L^{\otimes 2})$ is not identically zero. Since $\deg(L^{\otimes-2}\otimes K_{X_0})\geq 2g-2$, $H^0(L^{\otimes-2}\otimes K_{X_0})\ne0$. Note that for any non-zero $\beta\in H^0(L^{\otimes-2}\otimes K_{X_0})$, we have $0\ne \alpha\cdot\beta\in H^0(K_{X_0}^{\otimes 2})$. Thus by \eqref{pefp}, there exists $[\eta]\in H^1(T_{X_0})$ such that $[\eta(\alpha\cdot\beta)]\ne0$ in $H^1(K_{X_0})$. By Serre duality, we must have $[\eta(\alpha)]\ne0$ in $H^1(L^{\otimes 2})$.
\end{proof}

\subsection{Rank 3 case}\label{sec_rk3}
By Theorem \ref{thm-A} (ii), we can prove Theorem \ref{main C} (ii) by proving
\begin{proposition}\label{rk3}Let $(E,\bar\p,\theta)$ be a stable rank 3 Higgs bundle over a compact Riemann surface $X_0$ of genus $g\geq 2$. If $\theta\ne0$, the induced map $\theta_\ast:H^1(X_0,T_{X_0})\to\Hy$ is always non-zero. 
\end{proposition}
\begin{proof}By Lemma \ref{redsl}, we may assume $\mathrm{tr}(\theta)$ is zero.

{\bf Case I:} $\mathrm{tr}(\theta^2)\ne0\in H^0(K_{X_0}^{\otimes2})$. We have some $[\eta]\in H^1(T_{X_0})$ such that $[\eta\lrcorner \mathrm{tr}(\theta^2)]\ne0\in H^1(K_{X_0})$ by \eqref{pefp}. Thus we can prove  $[(0,\eta(\theta))]\in\Hy$ is non-zero by the same argument as the proof of {\bf Case I} in  Proposition \ref{rk2}.

{\bf Case II:} $\mathrm{tr}(\theta^2)=0,\ \det\theta\ne0$. Let $\Sigma_\theta$ be its spectral curve, which is reduced but could be reducible. By \eqref{BNRc}, we have a torsion free rank 1 coherent sheaf $\mathcal F$ on $\Sigma_\theta$ such that $\pi_\ast\mathcal F=E$ as sheaves on $X_0$. 

\noindent {\bf Claim:} $\theta_\ast$ factors through $H^1(\Sigma_\theta,\mathscr{H}\!om(\mathcal F,\mathcal F))$.
\begin{proof}[Proof of the Claim]

Note that $\theta:T_{X_0}\to \pi_\ast\mathscr{H}\!om(\mathcal F,\mathcal F)\to\End E $ induces $$\theta_\ast:H^1(X_0,T_{X_0})\to H^1(X_0,\pi_\ast\mathscr{H}\!om(\mathcal F,\mathcal F))=H^1(\Sigma_\theta,\mathscr{H}\!om(\mathcal F,\mathcal F)).$$
By Lemma \ref{A1} and Lemma \ref{Hitfib}, we have $$H^1(\Sigma_\theta,\mathscr{H}\!om(\mathcal F,\mathcal F))\hookrightarrow \mathrm{Ext}^1(\mathcal F,\mathcal F)\hookrightarrow\Hy.$$ This completes the proof of this claim.\end{proof}

Let $\nu:\tilde\Sigma_\theta\to\Sigma_\theta$ be the normalization of $\Sigma_\theta$. By Proposition \ref{normendo}, we have $$T_{X_0}\stackrel{\theta}{\rightarrow}\pi_\ast \mathscr{H}\!om(\mathcal F,\mathcal F)\stackrel{\iota}{\hookrightarrow} (\pi\circ\nu)_\ast\mathcal O_{\tilde\Sigma_\theta}.$$
So it suffices to prove that the induced map $(\iota\circ\theta)_\ast:H^1(T_{X_0})\to H^1((\pi\circ\nu)_\ast\mathcal O_{\tilde\Sigma_\theta})$ is non-zero.
In this case, the spectral cover $\pi:\Sigma_\theta\to X_0$ is a cyclic cover of degree $3$ defined by $\det\theta\in H^0(K_{X_0}^{\otimes 3})$.
Let $D:=\mathrm{div}(\det\theta)$. By \cite[Corollary 3.11]{EV}, we have the following isomorphism \begin{align} (\pi\circ\nu)_\ast\mathcal O_{\tilde \Sigma_\theta}=\bigoplus\limits_{i=0}^{2}\mathcal L^{(i)^{-1}},\end{align}
where $\mathcal L^{(i)^{-1}}=T_{X_0}^{\otimes i}([\frac{i}{3}\cdot D]). $ So we have the following short exact sequence
$$0\to T_{X_0}\stackrel{\iota\circ\theta}{\rightarrow}\bigoplus\limits_{i=0}^{2}\mathcal L^{(i)^{-1}}\to \mathcal O_{X_0}\oplus \mathcal O_{[\frac{1}{3}\cdot D]}\oplus \mathcal L^{(2)^{-1}}\to 0$$
Taking the cohomology and by Lemma \ref{H1Ker}, we have
$$0\to \bigoplus\limits_{i=0}^{2}H^0(\mathcal L^{(i)^{-1}})\to H^0(\mathcal O_{X_0}\oplus \mathcal O_{[\frac{1}{3}\cdot D]}\oplus \mathcal L^{(2)^{-1}})\to H^1(T_{X_0})\stackrel{(\iota\circ\theta)_\ast}{\rightarrow} H^1(\Ker(\tad))\to\cdots$$
If $\dim H^0(\mathcal O_{[\frac{1}{3}\cdot D]})=\deg([\frac{1}{3}\cdot D])< \dim H^1(T_{X_0})$, $(\iota\circ\theta)_\ast$ is not a zero map. But $\deg D=6g-6$, and thus $\deg([\frac{1}{3}\cdot D])< \dim H^1(T_{X_0})$.

{\bf Case III:} $\mathrm{tr}(\theta^2)=0,\ \det\theta=0$, i.e. the nilpotent case. By Lemma \ref{limcase}, we just need to prove this proposition for non-unitary graded Higgs bundles. And we will prove it case by case.

(i) There exists a rank 2 holomorphic vector bundle $F$ and a line bundle $L$ such that $E=F\oplus L, \theta|_F=\alpha\in H^0(F^\vee\otimes L\otimes K_{X_0})\backslash\{0\}$ and $\theta|_L\equiv 0$. Thus $\deg F=-\deg L>0$. Since $L$ is a line bundle, there exists a sub line bundle $L_1$ of $F$ such that $\theta|_{L_1}\equiv 0$, which implies $\deg L_1<0$. Then we have the following short exact sequence
$$0\to L_1\stackrel{i}{\to} F\to Q\to 0,$$
where $Q$ is the quotient line bundle. After tensoring this sequence with $L^\vee\otimes K_{X_0}$ and taking the cohomology, we have $$0\to H^0(L_1\otimes L^\vee\otimes K_{X_0})\stackrel{i_\ast}{\to} H^0(F\otimes L^\vee\otimes K_{X_0})\to H^0(Q\otimes L^\vee\otimes K_{X_0})\to \cdots.$$
By the Riemann-Roch theorem, we have \begin{align}\label{ineq1}h^0(F\otimes L^\vee\otimes K_{X_0})\geq 2g-2+\deg F+2\deg L^\vee>0.\end{align} We prove that the above $i_\ast:H^0(L_1\otimes L^\vee\otimes K_{X_0})\to H^0(F\otimes L^\vee\otimes K_{X_0})$ is not surjective. 
\begin{itemize}\item If $\deg(L_1\otimes L^\vee)\leq1-2g$. Then $H^0(L_1\otimes L^\vee\otimes K_{X_0})=0$ and $i_\ast$ is clearly not surjective. 
\item If $\deg(L_1\otimes L^\vee)\in[2-2g,0]$. By the Riemann-Roch theorem and the Clifford's inequality, we have $$h^0(L_1\otimes L^\vee\otimes K_{X_0})=h^0(L_1^\vee\otimes L)+\deg(L_1\otimes L^\vee\otimes K_{X_0})-(g-1)\leq g+\cc\deg(L^\vee\otimes L_1).$$
By \eqref{ineq1} and $\deg L_1<0$, we have $h^0(L_1\otimes L^\vee\otimes K_{X_0})<h^0(F\otimes L^\vee\otimes K_{X_0}).$
\item If $\deg(L_1\otimes L^\vee)>0$. By the Riemann-Roch theorem, we have $$h^0(L_1\otimes L^\vee\otimes K_{X_0})=\deg(L_1\otimes L^\vee\otimes K_{X_0})-(g-1).$$
By \eqref{ineq1} and $\deg L_1<0$, we have $h^0(L_1\otimes L^\vee\otimes K_{X_0})<h^0(F\otimes L^\vee\otimes K_{X_0}).$
\end{itemize}
Then there exists $\beta\in H^0(F\otimes L^\vee\otimes K_{X_0})$ such that $\alpha\cdot\beta\ne0\in H^0(K_{X_0}^{\otimes2})$. Hence  there exists $[\eta]\in H^1(T_{X_0})$ such that $[\eta(\alpha\cdot\beta)]\ne0\in H^1(K_{X_0})$, which implies $[\eta(\alpha)]\ne0\in H^1(F^\vee\otimes L)$. By Lemma \ref{crigra}, we must have $\theta_\ast\ne0$ in this case.

(ii) There exists a rank 2 holomorphic vector bundle $F$ and a line bundle $L$ such that $E=F\oplus L, \theta|_L=\alpha\in H^0(F\otimes L^\vee\otimes K_{X_0})\backslash\{0\}$ and $\theta|_F\equiv 0$. Thus $-\deg F=\deg L>0$. Since $L$ is a line bundle, there exists a sub line bundle $L_1$ of $F$ such that $\theta|_{L}$ factors through $L_1\otimes K_{X_0}$. Since $\theta|_{L_1}\equiv 0$, we also have $\deg L_1<0$. We have the following short exact sequence
$0\to L_1\to F\to Q\to 0,$ 
where $Q$ is the quotient line bundle such that $\deg Q>\deg F$. Then we have $$0\to Q^\vee\otimes L\otimes K_{X_0}\stackrel{j}{\to}F^\vee\otimes L\otimes K_{X_0}\to L_1^\vee\otimes L\otimes K_{X_0}\to 0.$$
By taking the cohomology, we have $$0\to H^0(Q^\vee\otimes L\otimes K_{X_0})\stackrel{j_\ast}{\to} H^0(F^\vee\otimes L\otimes K_{X_0})\to H^0(L_1^\vee\otimes L\otimes K_{X_0})\to \cdots.$$
By a similar argument as in (i), we can prove that the above $j_\ast:H^0(Q^\vee\otimes L\otimes K_{X_0}){\to} H^0(F^\vee\otimes L\otimes K_{X_0})$ is not surjective. Then we can prove $\theta_\ast\ne0$ using the fact that $j_\ast$ is not surjective.

(iii) There exist holomorphic line bundles $L_1,L_2$ and $L_3$ such that $E=L_1\oplus L_2\oplus L_3$ and $\theta|_{L_1}=\alpha\in H^0(L_1^\vee\otimes L_2\otimes K_{X_0}), \theta|_{L_2}=\beta\in H^0(L_2^\vee\otimes L_3\otimes K_{X_0})$ and $\theta|_{L_3}\equiv 0$. Since $\theta\ne0$, we may assume both $\alpha$ and $\beta$ are non-zero (if one of them is zero, we go back to (i) or (ii)). We claim that either $H^0(L_1\otimes L_2^\vee\otimes K_{X_0})$ or $H^0(L_2\otimes L_3^\vee\otimes K_{X_0})$ is non-zero. Note that $\deg L_3<0$ and $\deg L_1+\deg L_2+\deg L_3=0$. Thus $\deg L_1-\deg L_2$ and $\deg L_2-\deg L_3$ cannot be both negative (otherwise $\deg L_1<\deg L_2<\deg L_3<0)$. This completes the proof of our claim. Now we prove that $\theta_\ast\ne0$ in this case.
\begin{itemize}
\item If $H^0(L_1\otimes L_2^\vee\otimes K_{X_0})\ne0$.  Then there exists $\gamma\ne0\in H^0(L_1\otimes L_2^\vee\otimes K_{X_0})$ such that $\alpha\cdot\gamma\ne0\in H^0(K_{X_0}^{\otimes 2})$, which implies there exists $[\eta]\in H^1(T_{X_0})$ such that $[\eta(\alpha)]\ne0\in H^1(L_1^\vee\otimes L_2)$. Hence $\theta_\ast\ne0$ by Lemma \ref{crigra}.
\item If $H^0(L_2\otimes L_3^\vee\otimes K_{X_0})\ne0$. We can prove $\theta_\ast\ne0$ similarly.\qedhere
\end{itemize}
\end{proof}
\begin{remark} When $g\geq 3$, one can verify that the proof of Proposition \ref{rk3} works for polystable non-unitary rank 3 Higgs bundles.
\end{remark}

\subsection{Global holomorphicity implies unitary or nilpotent}\label{sec_globolhol}

We keep the notation of this section.  Thus $X_0$ is a compact Riemann surface of genus
$g\geq 2$, $(E,\bar\partial,\theta)$ is a degree-zero polystable Higgs bundle on $X_0$, and
\[
\sigma:\mathcal T_g\longrightarrow
\mathcal M_{\mathrm{Dol}}(\mathcal X/\mathcal T_g)
\]
is its isomonodromic deformation.  

By Lemma~\ref{redsl}, global holomorphicity also forces $\mathrm{tr}(\theta)=0$.  Hence, we may work with the trace-free
endomorphism bundle $\End_0E$ and the trace-free centralizer
\begin{equation}\label{global-Ztheta}
Z_\theta:=\Ker\!\left(\mathrm{ad}(\theta):\End_0E\longrightarrow
\End_0E\otimes K_{X_0}\right).
\end{equation}
The trace-free deformation complex is a direct summand of the full deformation complex, and
the proof of Lemma~\ref{H1Ker} gives an injection
\begin{equation}\label{global-centralizer-injection}
H^1(X_0,Z_\theta)\hookrightarrow
\mathbb H^1\bigl(X_0,(\End_0E,\mathrm{ad}(\theta))\bigr).
\end{equation}
When $\mathrm{tr}(\theta)=0$, we denote by
\begin{equation}\label{global-centralizer-map}
\widetilde\theta_*:H^1(X_0,T_{X_0})\longrightarrow H^1(X_0,Z_\theta)
\end{equation}
the map induced by $\theta:T_{X_0}\to Z_\theta$.  Thus
$\theta_*$ factors through \eqref{global-centralizer-map}, and the injection
\eqref{global-centralizer-injection} implies
\begin{equation}\label{global-rank-centralizer-map}
\operatorname{rank}\theta_*=\operatorname{rank}\widetilde\theta_*.
\end{equation}

\subsubsection{Clifford type estimate and a numerical contraction theorem}
Let $f:Y\to X_0$ be a finite nonconstant morphism between connected smooth projective curves,
of degree $d$, and let
\[
0\neq w\in H^0(Y,f^*K_{X_0}).
\]
Define the contraction map
\begin{equation}\label{global-contraction-map}
\Phi_w:H^1(X_0,T_{X_0})\longrightarrow H^1(Y,\mathcal O_Y),
\qquad [\mu]\longmapsto[(f^*\mu)\lrcorner w].
\end{equation}
\noindent\textbf{Question:} Is this $\Phi_w$ a nonzero map? We will prove this holds under some numerical condition in Theorem~\ref{global-numerical-contraction} and use it to study Conjecture~\ref{conj-F}.

\vspace{.3cm}

Put $A=f_*\mathcal O_Y$.  Finite duality and the projection formula identify $w$ with a
section $s_w\in H^0(X_0,A\otimes K_{X_0})$, and the Serre dual of \eqref{global-contraction-map}
is
\begin{equation}\label{global-dual-contraction}
m_w:H^0(X_0,A^\vee\otimes K_{X_0})\longrightarrow
H^0(X_0,K_{X_0}^{\otimes2}),\qquad u\longmapsto u(s_w).
\end{equation}
Equivalently, under $f_*K_Y\simeq A^\vee\otimes K_{X_0}$, it is the trace-product map
$\eta\mapsto\mathrm{tr}_{Y/X_0}(w\eta)$.  Therefore
\begin{equation}\label{global-rank-duality}
\operatorname{rank}\Phi_w=\operatorname{rank}m_w.
\end{equation}

The following Clifford type result is due to Landesman and Litt \cite{LL}.

\begin{lemma}[Clifford-type multiplication estimate]\label{global-LL-lemma}
Let $Q$ be a nonzero semistable vector bundle on $X_0$ with $\mu(Q)\leq0$, and let
$0\neq v\in H^0(X_0,Q\otimes K_{X_0})$.  Set
\[
M_v:H^0(X_0,Q^\vee\otimes K_{X_0})\longrightarrow
H^0(X_0,K_{X_0}^{\otimes2}),\qquad u\longmapsto u(v).
\]
Then
\begin{equation}\label{global-LL-weak}
\operatorname{rank}M_v\geq g-\operatorname{rank}Q.
\end{equation}
If $\mu(Q)<0$, then
\begin{equation}\label{global-LL-strict}
\operatorname{rank}M_v>g-\operatorname{rank}Q.
\end{equation}
\end{lemma}

\begin{proof}
The bundle $Q^\vee\otimes K_{X_0}$ is semistable of slope
$2g-2-\mu(Q)\geq2g-2$.  The section $v$ induces a nonzero morphism
$Q^\vee\otimes K_{X_0}\to K_{X_0}^{\otimes2}$ whose kernel is a subbundle $U$.
Moreover,
\[
h^0(Q^\vee\otimes K_{X_0})-h^0(U)=\operatorname{rank}M_v.
\]
Apply \cite[Proposition 5.2.4]{LL} with empty parabolic divisor, $c=1$, and
$\delta=\operatorname{rank}M_v$.  Part (II) gives \eqref{global-LL-weak} at slope $2g-2$,
and part (I) gives \eqref{global-LL-strict} above that slope.
\end{proof}

\begin{lemma}\label{global-direct-image-slopes}The unit and trace maps
give a splitting
\begin{equation}\label{global-trace-splitting}
A=\mathcal O_{X_0}\cdot1\oplus A_0,
\qquad A_0:=\Ker(\mathrm{tr}_{Y/X_0}),
\qquad \operatorname{rank}A_0=d-1,
\end{equation}
and every subbundle of $A_0$ must have slope at most zero.
\end{lemma}

\begin{proof}The splitting \eqref{global-trace-splitting} holds by definition.
 By the Harder-Narasimhan theorem, it suffices to prove the claim for any semistable subbundle.

Suppose that a semistable bundle $F$ of positive slope admits a nonzero map
$F\to f_*\mathcal O_Y$.  By adjunction there is a nonzero map $f^*F\to\mathcal O_Y$.
Finite pullback preserves semistability, so $f^*F$ is semistable of
positive slope, which cannot map nontrivially to the slope-zero line bundle $\mathcal O_Y$.
The assertion for $A_0$
follows because it is a direct summand.
\end{proof}

\begin{theorem}[Numerical contraction theorem]\label{global-numerical-contraction}
For every finite map $f:Y\to X_0$ of degree $d$ and every nonzero
$w\in H^0(Y,f^*K_{X_0})$,
\begin{equation}\label{global-contraction-bound}
\operatorname{rank}\Phi_w\geq\max\{0,g-d+1\}.
\end{equation}
In particular, if $d\leq g$, then $\Phi_w\neq0$.
\end{theorem}

\begin{proof}
Using \eqref{global-trace-splitting}, write
\[
s_w=\alpha\cdot1+s_w^0,
\qquad \alpha\in H^0(X_0,K_{X_0}),
\qquad s_w^0\in H^0(X_0,A_0\otimes K_{X_0}).
\]
If $\alpha\neq0$, the restriction of $m_w$ to the dual of the trivial summand is
multiplication by $\alpha$,
\[
H^0(X_0,K_{X_0})\longrightarrow H^0(X_0,K_{X_0}^{\otimes2}),
\qquad\beta\longmapsto\alpha\beta,
\]
which is injective.  Hence $\operatorname{rank}\Phi_w\geq g$ by
\eqref{global-rank-duality}.

Assume now that $\alpha=0$.  Let
\[
0=F_0\subset F_1\subset\cdots\subset F_m=A_0
\]
be the Harder--Narasimhan filtration, with semistable quotients $Q_i=F_i/F_{i-1}$. Choose the
least $j$ such that $s_w\in H^0(X_0,F_j\otimes K_{X_0})$, and let
$0\neq v\in H^0(X_0,Q_j\otimes K_{X_0})$ be its image.  By the slope property of the 
Harder--Narasimhan filtration, $H^0(X_0,A_0/F_j)=0$.  By Serre duality,
$H^1(X_0,(A_0/F_j)^\vee\otimes K_{X_0})=0$, and hence
\[
H^0(X_0,A_0^\vee\otimes K_{X_0})\twoheadrightarrow
H^0(X_0,F_j^\vee\otimes K_{X_0}).
\]
Since $Q_j^\vee\hookrightarrow F_j^\vee$, every section occurring in $M_v$ lifts to the
domain of $m_w$.  Thus $\operatorname{rank}M_v\leq\operatorname{rank}m_w$.  Applying
Lemma~\ref{global-LL-lemma} gives
\[
\operatorname{rank}m_w\geq g-\operatorname{rank}Q_j
\geq g-(d-1)=g-d+1.
\]
Together with \eqref{global-rank-duality}, this proves \eqref{global-contraction-bound}.
\end{proof}
\begin{remark}We expect this Theorem~\ref{global-numerical-contraction} holds for any nonzero $w$ and $g(X_0)\geq 3$ without the numerical condition $d\leq g=g(X_0)$. For the genus $g=2$ case, there exists a counterexample with $d=36$ given in Section~\ref{sec_counterexample}.
\end{remark}

\subsubsection{Reduced spectral curves}
Assume that the spectral curve $\pi:\Sigma_\theta\to X_0$ is reduced. By the BNR correspondence, there is a rank-one torsion-free sheaf $\mathcal F$ on
$\Sigma_\theta$ such that
\[
(E,\theta)\simeq\pi_*(\mathcal F,\lambda),
\]
where $\lambda$ is the tautological one-form; see \cite{BNR,Daniel}.

Let $\Sigma_1,\ldots,\Sigma_r$ be the irreducible components which are not the zero section.
Let $\nu_i:Y_i\to\Sigma_i$ be the normalization and put
\begin{equation}\label{global-spectral-component-data}
f_i:=\pi\circ\nu_i:Y_i\longrightarrow X_0,
\qquad d_i:=\deg f_i,
\qquad w_i:=\nu_i^*\lambda\in H^0(Y_i,f_i^*K_{X_0}).
\end{equation}
Then $w_i\neq0$.

\begin{lemma}[Spectral-component factorization]\label{global-spectral-factorization}
Assume $\mathrm{tr}(\theta)=0$.  For every nonzero component $\Sigma_i$, there is a natural morphism $q_i:Z_\theta\longrightarrow(f_i)_*\mathcal O_{Y_i}$ such that $(q_i\otimes\operatorname{id}_{K_{X_0}})(\theta)=w_i$.  Taking first cohomology, we get
$\Phi_{w_i}=H^1(q_i)\circ\widetilde\theta_*.$ Consequently,
\begin{equation}\label{global-spectral-rank-comparison}
\operatorname{rank}\theta_*\geq\operatorname{rank}\Phi_{w_i}.
\end{equation}
\end{lemma}

\begin{proof}The spectral module description identifies
\[
\Ker\!\left(\mathrm{ad}(\theta):\End E\to\End E\otimes K_{X_0}\right)\simeq\pi_*\mathscr{H}\!om_{\Sigma_\theta}(\mathcal F,\mathcal F).
\]
By Proposition~\ref{normendo}, normalization gives
\[
\mathscr{H}\!om_{\Sigma_\theta}(\mathcal F,\mathcal F)
\hookrightarrow\nu_*\mathcal O_{\widetilde\Sigma_\theta}.
\]
Projecting to the component $Y_i$, pushing forward to $X_0$, and restricting to the
trace-free summand gives $q_i$. Hence $(q_i\otimes\operatorname{id}_{K_{X_0}})(\theta)=w_i$. For any $\mu$ in $H^1(X_0,T_{X_0})$, we have
\[
H^1(q_i)(\widetilde\theta_*[\mu])
=\bigl[(f_i^*\mu)\lrcorner w_i\bigr]
=\Phi_{w_i}([\mu]).
\]
The estimate 
\eqref{global-spectral-rank-comparison} follows from
\eqref{global-rank-centralizer-map}.
\end{proof}

Let $\epsilon=1$ if the zero section is an irreducible component of $\Sigma_\theta$, and let
$\epsilon=0$ otherwise.  Reducedness implies
\begin{equation}\label{global-component-degree-sum}
n:=\mathrm{rank}E=\epsilon+\sum_{i=1}^r d_i.
\end{equation}

\begin{theorem}[Reduced-spectrum non-unitary theorem]
\label{global-reduced-spectrum-theorem}
Assume that the spectral curve $\Sigma_\theta$ is reduced and that $n\leq g$.  If
$(E,\bar\partial,\theta)$ is non-unitary, then its isomonodromic deformation over
$\mathcal T_g$ is not globally holomorphic.
\end{theorem}

\begin{proof}
Suppose, to the contrary, that the isomonodromic section is globally holomorphic.  By
Theorem~\ref{thm-A}(ii) and Lemma~\ref{redsl}, one has $\theta_*=0$ and
$\mathrm{tr}(\theta)=0$.  Since the Higgs bundle is non-unitary, $\theta\neq0$.  Hence the
reduced spectral curve has a component $\Sigma_i$ which is not the zero section.  Its degree
satisfies $d_i\leq n\leq g$.  Theorem~\ref{global-numerical-contraction} gives
$\Phi_{w_i}\neq0$, while Lemma~\ref{global-spectral-factorization} gives
\[
\operatorname{rank}\theta_*\geq\operatorname{rank}\Phi_{w_i}>0,
\]
a contradiction.
\end{proof}

\subsubsection{Any case without reduced assumption}
We impose no condition on the spectral curve.

\begin{lemma}\label{global-max-centralizer}
If $0\neq A\in\mathfrak{sl}_n(\mathbb C)$, then the dimension of its centralizer algebra satisfies
\begin{equation}\label{global-max-centralizer-bound}
\dim Z_{\mathfrak{sl}_n}(A)\leq(n-1)^2.
\end{equation}
\end{lemma}

\begin{proof}
For a Jordan form, the dimension of the centralizer in $\mathfrak{gl}_n$ is
\[
\sum_a\sum_j\bigl(m_j^{(a)}\bigr)^2,
\]
where $m_j^{(a)}$ denotes the number of Jordan blocks with eigenvalue
$a$ whose size is at least $j$.  Among
non-scalar matrices this is at most $(n-1)^2+1$.  Since the identity belongs to the
$\mathfrak{gl}_n$-centralizer, intersecting with $\mathfrak{sl}_n$ lowers the dimension by one.
\end{proof}

\begin{theorem}[Non-unitary theorem]\label{global-centralizer-theorem}
Assume $\mathrm{tr}(\theta)=0$ and $\theta\neq0$.  Then
\begin{equation}\label{global-centralizer-bound}
\operatorname{rank}\theta_*\geq\max\{0,g-\operatorname{rank}Z_\theta\}.
\end{equation}
Consequently, if
\begin{equation}\label{global-unitary-threshold-equation}
g>(n-1)^2,
\end{equation}
then every non-unitary rank-$n$ Higgs bundle has non-holomorphic isomonodromic deformation
 over $\mathcal T_g$.
\end{theorem}

\begin{proof}
The adjoint Higgs bundle $(\End_0E,\mathrm{ad}(\theta))$ is polystable of degree zero; see
\cite{Simp92}. Thus the slope of every subbundle of
$Z_\theta$ is at most zero.

By Serre duality, the dual of \eqref{global-centralizer-map} is
\begin{equation}\label{global-centralizer-multiplication}
m_\theta:H^0(X_0,Z_\theta^\vee\otimes K_{X_0})\longrightarrow
H^0(X_0,K_{X_0}^{\otimes2}),\qquad u\longmapsto u(\theta),
\end{equation}
and \eqref{global-rank-centralizer-map} gives
$\operatorname{rank}\theta_*=\operatorname{rank}m_\theta$.

Choose the first Harder--Narasimhan quotient $Q_j$ of $Z_\theta$ which detects
$\theta\in H^0(X_0,Z_\theta\otimes K_{X_0})$, and let
$0\neq v\in H^0(X_0,Q_j\otimes K_{X_0})$ be its image.  The lifting argument used in the
proof of Theorem~\ref{global-numerical-contraction} gives
\[
\operatorname{rank}m_\theta\geq\operatorname{rank}M_v
\geq g-\operatorname{rank}Q_j\geq g-\operatorname{rank}Z_\theta
\]
by Lemma~\ref{global-LL-lemma}.  This proves \eqref{global-centralizer-bound}.

Now assume \eqref{global-unitary-threshold-equation} and suppose that a non-unitary Higgs
bundle had globally holomorphic isomonodromic deformation. 
Theorem~\ref{thm-A}(ii) and Lemma~\ref{redsl} would give $\theta_*=0$, $\mathrm{tr}(\theta)=0$, and $\theta\neq0$.
Applying Lemma~\ref{global-max-centralizer} gives
$\operatorname{rank}Z_\theta\leq(n-1)^2<g$, whereas \eqref{global-centralizer-bound} gives $\theta_*\neq0$.
\end{proof}

The nilpotency conclusion admits a better threshold than the vanishing conclusion above.

\begin{lemma}\label{global-quadratic-centralizer}
Let $A\in\mathfrak{sl}_n(\mathbb C)$ be non-nilpotent and satisfy $\mathrm{tr}(A^2)=0$.  Then
\begin{equation}\label{global-quadratic-centralizer-bound}
\dim Z_{\mathfrak{sl}_n}(A)\leq(n-2)^2+1.
\end{equation}
\end{lemma}

\begin{proof}
For $n=2$ there is no such matrix.  For $n\geq3$, let $a_1,\ldots,a_s$ be the distinct
eigenvalues and let $n_i$ be the dimensions of the corresponding generalized eigenspaces.
There cannot be only one eigenvalue, since $A$ is trace free and non-nilpotent.
There cannot be exactly two eigenvalues either by the fact $\operatorname{tr}(A)=\operatorname{tr}(A^2)=0$.  Hence $s\geq3$.  The $\mathfrak{gl}_n$-centralizer has dimension at
most
\[
\sum_{i=1}^s n_i^2\leq(n-2)^2+2,
\]
the maximum occurring for $(n-2,1,1)$.  Passing to $\mathfrak{sl}_n$ proves \eqref{global-quadratic-centralizer-bound}.
\end{proof}

\begin{theorem}[Non-nilpotent theorem]\label{global-nilpotency-threshold}If
\begin{equation}\label{global-nilpotency-threshold-equation}
g>(n-2)^2+1,
\end{equation}
every non-nilpotent rank-$n$ Higgs bundle has non-holomorphic isomonodromic
deformation.
\end{theorem}

\begin{proof}
Suppose that the deformation is globally holomorphic.  
Theorem~\ref{thm-A}(ii) and Lemma~\ref{redsl} give $\theta_*=0$ and $\mathrm{tr}(\theta)=0$.  By a similar argument as in the proof of 
Proposition~\ref{rk2}, $\mathrm{tr}(\theta^2)=0$.  If $\theta$ were
non-nilpotent, Lemma~\ref{global-quadratic-centralizer} would give
$\operatorname{rank}Z_\theta\leq(n-2)^2+1<g$.  Theorem~\ref{global-centralizer-theorem} would then imply
$\theta_*\neq0$, a contradiction.
\end{proof}

\subsection{A genus-two counterexample to Conjecture~\ref{conj-F}}\label{sec_counterexample}
\label{genus-two-conjecture-F-counterexample}
 The following construction is essentially due to Bogomolov--Tschinkel and Landesman--Litt; see \cite{LLPW}. Let
$\pi:\mathcal X\to\mathcal T_2$ be the universal family. There exists a finite \mbox{\'{e}tale} cover $q:\mathcal Y\to\mathcal X$ of degree $36$ such that we have another finite cover
\[
h:\mathcal Y\longrightarrow E_0\times\mathcal T_2,
\]
where $E_0$ is a fixed elliptic curve, and every
fiber map $h_t:Y_t\to E_0$ has degree $4$; see
\cite[Proposition~6.7 and Theorem~6.8]{LLPW}.

Choose $0\neq\alpha\in H^0(E_0,K_{E_0})$ and set
\[
w=h^*\alpha\in H^0(\mathcal Y,K_{\mathcal Y/\mathcal T_2}).
\]
Since $q$ is relative \mbox{\'{e}tale}, multiplication by $w$ defines the relative Higgs bundle
\begin{equation}\label{degree-36-relative-Higgs}
(\mathcal E,\Theta)
:=\bigl(q_*\mathcal O_{\mathcal Y},\ q_*m_w\bigr)
\end{equation}
of rank $36$ on $\mathcal X/\mathcal T_2$.

\begin{proposition}
\label{degree-36-conjecture-F-counterexample}
The family \eqref{degree-36-relative-Higgs} is a holomorphic isomonodromic family of
polystable degree-zero Higgs bundles, and every Higgs field $\Theta_t$ is non-nilpotent.
Consequently, Conjecture~\ref{conj-F} fails for $g=2$.
\end{proposition}

\begin{proof}
The form $\alpha$ determines a fixed rank-one representation of $\pi_1(E_0)$.  Its pullback by
$h_t$ is constant under the markings of the family $\mathcal Y/\mathcal T_2$ by integration, denoted by $\chi$.  If
$H=\pi_1(Y_t)\subset G=\pi_1(X_t)$ is the subgroup determined by the degree-$36$ covering,
the direct-image local system has monodromy $\operatorname{Ind}_H^G(\chi),$ which is independent of $t$.  Hence the family is isomonodromic.  

The induced representation
is semisimple, so the corresponding Higgs bundles are polystable of degree zero by the
Hitchin--Simpson correspondence.  On the other hand, $\mathcal E=q_*\mathcal O_{\mathcal Y}$
and $\Theta=q_*m_w$ are holomorphic over $\mathcal T_2$, so its Dolbeault section is
holomorphic.

Finally, on an open set over which $q_t$ splits into $36$ sheets,
\[
\Theta_t=\operatorname{diag}(w_{t,1},\ldots,w_{t,36}).
\]
Since $h_t$ is nonconstant and $\alpha\neq0$, one has $w_t=h_t^*\alpha\neq0$; hence at least
one eigen-1-form is nonzero and $\Theta_t$ is not nilpotent.
\end{proof}

\begin{remark}
Fix any point $[X_0]\in\mathcal T_2$, we have the finite \mbox{\'{e}tale} cover $q_0:Y_0\to X_0$ defined above. We remark the form $w_0:=h_0^*\alpha\in H^0(K_{Y_0})=H^0(q_0^*K_{X_0})$ gives a zero contraction map $\Phi_{w_0}$ defined in \eqref{global-contraction-map}. This means we cannot expect Theorem~\ref{global-numerical-contraction} to be true when $g=2$ and without assuming any numerical condition.
\end{remark}

\appendix
\section{Local--global spectral sequence}
For self-containedness, we review the standard local-global spectral sequence. For any two coherent sheaves \( \mathcal F,\mathcal G\) on a scheme \( X \). The Ext sheaf
\(\mathscr{E}\!xt^n(\mathcal F, \mathcal G)\) is defined to be the right derived functor of the sheaf \(\mathscr{H}\!om(\mathcal F, -)\). The Ext group \(\mathrm{Ext}^n(\mathcal F, \mathcal G)\) is defined to be the right derived functor of \(\mathrm{Hom}(\mathcal F, -)\).  

The Grothendieck spectral sequence for the composition \(\Gamma \circ \mathscr{H}\!om(\mathcal F, -)\) gives the local-to-global spectral sequence (see \cite[Chapter 29.2]{Risingsea})

\begin{align}\label{loctogol}
E_2^{p,q} = H^p(X, \mathscr{E}\!xt^q(\mathcal F, \mathcal G)) \;\Rightarrow\; \mathrm{Ext}^{p+q}(\mathcal F, \mathcal G).
\end{align}
 
\begin{lem}\label{A1} Let $X$ be a projective Noetherian scheme of dimension 1. We have the following short exact sequence
$$0\to  H^1(X, \mathscr{H}\!om(\mathcal F, \mathcal F))\to \mathrm{Ext}^1(\mathcal F, \mathcal F)\to H^0(X, \mathscr{E}\!xt^1(\mathcal F, \mathcal F))\to 0$$
\end{lem}
\begin{proof}For a first quadrant spectral sequence $E^{p,q}\;\Rightarrow\; E^{p+q}$, we have the 5-term exact sequence \[
0 \to E_2^{1,0} \to E^1 \to E_2^{0,1} \to E_2^{2,0} \to E^2.
\]
Applying this to the local--global spectral sequence and note that $E_2^{2,0}$ is zero in this case.\end{proof}

\section{Normalization and the endomorphism sheaf}
We are grateful to Jinbang Yang for helpful discussions about the following proposition and its proof.
\begin{prop}\label{normendo}
Let $X$ be a reduced quasi-projective scheme over $\mathbb C$, not necessarily irreducible, and let $\mathcal F$ be a coherent $\mathcal O_X$-module that is torsion-free and has rank one on every irreducible component of $X$. Let
\[
\nu:\widetilde X\longrightarrow X
\]
be the normalization. Then there are natural inclusions of coherent $\mathcal O_X$-algebras
\begin{equation}\label{Endo}
\mathcal O_X\hookrightarrow
\mathscr{E}\!nd_{\mathcal O_X}(\mathcal F)
\hookrightarrow
\nu_*\mathcal O_{\widetilde X}.
\end{equation}
\end{prop}

\begin{proof}
The assertion is local on $X$. Let $U=\operatorname{Spec}A\subset X$ be an affine open subset, and let $M$ be the finite $A$-module corresponding to $\mathcal F|_U$. Since $A$ is reduced, the set $S$ of non-zero-divisors of $A$ is multiplicatively closed, and the total ring of fractions
\[
Q(A):=S^{-1}A
\]
is a finite product of fields, one for each irreducible component of $U$. The natural map $A\to Q(A)$ is injective.

Because $M$ is torsion-free and has rank one on every irreducible component, the localization map is injective and gives
\[
M\hookrightarrow M\otimes_A Q(A)\simeq Q(A).
\]
Localizing an $A$-linear endomorphism of $M$ therefore yields an injective homomorphism
\[
\operatorname{End}_A(M)
\hookrightarrow
\operatorname{End}_{Q(A)}\bigl(M\otimes_A Q(A)\bigr)
\simeq
\operatorname{End}_{Q(A)}\bigl(Q(A)\bigr)
\simeq Q(A).
\]
Under this identification, an element $\varphi\in\operatorname{End}_A(M)$ is represented by some $q\in Q(A)$ satisfying
\[
qM\subseteq M.
\]

We claim that $q$ is integral over $A$. First observe that $M$ is a faithful $A$-module. Indeed, if $aM=0$ for some $a\in A$, then after tensoring with $Q(A)$ we obtain $aQ(A)=0$. Since $A\hookrightarrow Q(A)$, this implies $a=0$.

Choose generators $m_1,\ldots,m_r$ of $M$ over $A$. Since $qM\subseteq M$, there are coefficients $a_{ij}\in A$ such that
\[
qm_i=\sum_{j=1}^r a_{ij}m_j,
\qquad i=1,\ldots,r.
\]
Let $B=(a_{ij})\in M_r(A)$ and let $m=(m_1,\ldots,m_r)^t$. Thus
\[
\det(qI_r-B)m_i=0
\qquad\text{for every }i.
\]
Since the $m_i$ generate $M$, the element $\det(qI_r-B)\in Q(A)$ annihilates $M$. After tensoring with $Q(A)$, it therefore annihilates
\[
M\otimes_A Q(A)\simeq Q(A),
\]
and hence $\det(qI_r-B)=0$ in $Q(A)$. Thus $q$ is integral over $A$.

Let $\overline A$ denote the integral closure of $A$ in $Q(A)$. The preceding argument proves
\[
\operatorname{End}_A(M)\subseteq\overline A.
\]
Since $X$ is quasi-projective, its normalization is finite, and $\nu^{-1}(U)=\operatorname{Spec}\overline A.$ Consequently,
\[
\operatorname{End}_A(M)
\subseteq
\Gamma\bigl(\nu^{-1}(U),\mathcal O_{\widetilde X}\bigr).
\]

Finally, multiplication defines a homomorphism
\[
A\longrightarrow\operatorname{End}_A(M).
\]
It is injective because $M$ is faithful. Sheafifying these local inclusions proves \eqref{Endo}.
\end{proof}

\thispagestyle{empty}


\begin{thebibliography}{99}
\bibitem{ACZ}
Y. Andr\'e, P. Corvaja, and U. Zannier,
\emph{The {B}etti map associated to a section of an abelian scheme},
Invent. Math., \textbf{222}, (2020), no.~1, 161--202.

\bibitem{Ati}
M. F. Atiyah,
\emph{Complex analytic connections in fibre bundles},
Trans. Amer. Math. Soc., \textbf{85}, (1957), 181--207.

\bibitem{BNR}
Arnaud Beauville, M. S. Narasimhan, and S. Ramanan,
\emph{Spectral curves and the generalised theta divisor},
J. Reine Angew. Math., \textbf{398}, (1989), 169--179.

\bibitem{biswas}
Indranil Biswas, Lynn Heller, and Sebastian Heller,
\emph{Holomorphic Higgs bundles over the Teichm\"uller space},
arxiv:2308.13860, (2023).

\bibitem{CGHX}
Serge Cantat, Ziyang Gao, Philipp Habegger, and Junyi Xie,
\emph{The geometric {B}ogomolov conjecture},
Duke Math. J., \textbf{170}, (2021), no.~2, 247--277.

\bibitem{CMP}
James Carlson, Stefan M\"uller-Stach, and Chris Peters,
\emph{Period mappings and period domains},
Cambridge Studies in Advanced Mathematics, vol.~168, Second edition, Cambridge University Press, Cambridge, 2017.

\bibitem{CT}
Ting Chen,
\emph{The associated map of the nonabelian {G}auss-{M}anin connection},
Cent. Eur. J. Math., \textbf{10}, (2012), no.~4, 1407--1421.

\bibitem{CTW}
Brian Collier, J\'er\'emy Toulisse, and Richard Wentworth,
\emph{Higgs bundle, isomonodromic leaves and minimal surfaces},
arxiv:2512.07152, (2025).

\bibitem{Corl}
Kevin Corlette,
\emph{Flat {$G$}-bundles with canonical metrics},
J. Differential Geom., \textbf{28}, (1988), no.~3, 361--382.

\bibitem{CMZ}
Pietro Corvaja, David Masser, and Umberto Zannier,
\emph{Torsion hypersurfaces on abelian schemes and {B}etti coordinates},
Math. Ann., \textbf{371}, (2018), no.~3-4, 1013--1045.

\bibitem{Don}
S. K. Donaldson,
\emph{Twisted harmonic maps and the self-duality equations},
Proc. London Math. Soc. (3), \textbf{55}, (1987), no.~1, 127--131.

\bibitem{EV}
H\'el\`ene Esnault and Eckart Viehweg,
\emph{Lectures on vanishing theorems},
DMV Seminar, vol.~20, Birkh\"auser Verlag, Basel, 1992.

\bibitem{FS}
Yixuan Fu and Mao Sheng,
\emph{Nonabelian Kodaira-Spencer maps},
arxiv:2509.06050, (2025).

\bibitem{Gri69}
Phillip A. Griffiths,
\emph{On the periods of certain rational integrals. {I}, {II}},
Ann. of Math. (2), \textbf{90}, (1969), 460--495; 90 (1969), 496--541.

\bibitem{Gri70}
Phillip A. Griffiths,
\emph{Periods of integrals on algebraic manifolds. {III}. {S}ome
              global differential-geometric properties of the period
              mapping},
Inst. Hautes \'Etudes Sci. Publ. Math., (1970), no.~38, 125--180.

\bibitem{Gui}
Olivier Guichard,
\emph{An introduction to the differential geometry of flat bundles
              and of {H}iggs bundles},
in The geometry, topology and physics of moduli spaces of {H}iggs
              bundles, Lect. Notes Ser. Inst. Math. Sci. Natl. Univ. Singap., vol.~36, pp.~1--63, World Sci. Publ., Hackensack, NJ, 2018.
              
\bibitem{THit}
Tam\'as Hausel and Nigel Hitchin,
\emph{Very stable {H}iggs bundles, equivariant multiplicity and
              mirror symmetry},
Invent. Math., \textbf{228}, (2022), no.~2, 893--989.

\bibitem{Hit}
N. J. Hitchin,
\emph{The self-duality equations on a {R}iemann surface},
Proc. London Math. Soc. (3), \textbf{55}, (1987), no.~1, 59--126.

\bibitem{HSYZ}
Tianzhi Hu, Ruiran Sun, Jinbang Yang, and Kang Zuo,
\emph{Isomonodromic deformations of Higgs bundles and characterization of the non-abelian Noether--Lefschetz locus},
arXiv:2606.18768, (2026).

\bibitem{HSZ}
Tianzhi Hu, Ruiran Sun, and Kang Zuo,
\emph{Kodaira-Spencer Map on the Hitchin-Simpson Correspondence},
arxiv:2511.14272, (2025).

\bibitem{HSZII}
Tianzhi Hu, Ruiran Sun, and Kang Zuo,
\emph{Non-existence of holomorphic isomonodromic deformation of a Higgs bundle},
arxiv:2512.15478, (2025).

\bibitem{KYZ}
Raju Krishnamoorthy, Jinbang Yang, and Kang Zuo,
\emph{Deformation theory of periodic Higgs-de Rham flows},
arxiv:2005.00579, (2020).

\bibitem{LLPW}
Aaron Landesman and Daniel Litt,
\emph{An introduction to the algebraic geometry of the Putman--Wieland conjecture},
European J. Math., \textbf{9}, (2023), Article 40.

\bibitem{LL}
Aaron Landesman and Daniel Litt,
\emph{Canonical representations of surface groups},
Ann. of Math. (2), \textbf{199}, (2024), no.~2, 823--897.

\bibitem{Li}
Qiongling Li,
\emph{An introduction to {H}iggs bundles via harmonic maps},
SIGMA Symmetry Integrability Geom. Methods Appl., \textbf{15}, (2019), Paper No. 035, 30.

\bibitem{LRY}
Kefeng Liu, Sheng Rao, and Xiaokui Yang,
\emph{Quasi-isometry and deformations of {C}alabi-{Y}au manifolds},
Invent. Math., \textbf{199}, (2015), no.~2, 423--453.

\bibitem{moch24}
Takuro Mochizuki,
\emph{Asymptotic behaviour of the Hitchin metric on the moduli space of Higgs bundles},
arxiv:2305.17638, (2024).

\bibitem{Daniel}
Daniel Schaub,
\emph{Courbes spectrales et compactifications de jacobiennes},
Math. Z., \textbf{227}, (1998), no.~2, 295--312.

\bibitem{Simp88}
Carlos T. Simpson,
\emph{Constructing variations of {H}odge structure using
              {Y}ang-{M}ills theory and applications to uniformization},
J. Amer. Math. Soc., \textbf{1}, (1988), no.~4, 867--918.

\bibitem{Simp92}
Carlos T. Simpson,
\emph{Higgs bundles and local systems},
Inst. Hautes \'Etudes Sci. Publ. Math., (1992), no.~75, 5--95.

\bibitem{Simp94I}
Carlos T. Simpson,
\emph{Moduli of representations of the fundamental group of a smooth
              projective variety. {I}},
Inst. Hautes \'Etudes Sci. Publ. Math., (1994), no.~79, 47--129.

\bibitem{Simp94II}
Carlos T. Simpson,
\emph{Moduli of representations of the fundamental group of a smooth
              projective variety. {II}},
Inst. Hautes \'Etudes Sci. Publ. Math., (1994), no.~80, 5--79.

\bibitem{Simp95}
Carlos Simpson,
\emph{The {H}odge filtration on nonabelian cohomology},
in Algebraic geometry---{S}anta {C}ruz 1995, Proc. Sympos. Pure Math., vol.~62, Part 2, pp.~217--281, Amer. Math. Soc., Providence, RI, 1997.

\bibitem{Simp10}
Carlos Simpson,
\emph{Iterated destabilizing modifications for vector bundles with
              connection},
in Vector bundles and complex geometry, Contemp. Math., vol.~522, pp.~183--206, Amer. Math. Soc., Providence, RI, 2010.

\bibitem{Tos}
Ognjen To\v{s}i\'c,
\emph{Non-strict plurisubharmonicity of energy on {T}eichm\"uller
              space},
Int. Math. Res. Not. IMRN, (2024), no.~9, 7820--7845.

\bibitem{UY}
K. Uhlenbeck and S.-T. Yau,
\emph{On the existence of {H}ermitian-{Y}ang-{M}ills connections in
              stable vector bundles},
in Comm. Pure Appl. Math., vol.~39, pp.~S257--S293, 1986.

\bibitem{Risingsea}
Ravi Vakil,
\emph{The rising sea---foundations of algebraic geometry},
Princeton University Press, Princeton, NJ, [2025] \copyright 2025.

\bibitem{Zuo}
Kang Zuo,
\emph{On the negativity of kernels of {K}odaira-{S}pencer maps on
              {H}odge bundles and applications},
in Asian J. Math., vol.~4, pp.~279--301, 2000.


\end{thebibliography}
\end{document}